\documentclass[11pt]{article}

\usepackage[margin=1.25in]{geometry}

\usepackage{amssymb,amsmath}
\usepackage{amsthm}
\usepackage{hyperref}
\usepackage{mathrsfs}
\usepackage{textcomp}
\hypersetup{
    colorlinks,%
    citecolor=black,%
    filecolor=black,%
    linkcolor=black,%
    urlcolor=black
}

\usepackage{verbatim}

\usepackage{sectsty}

\makeatletter

\newdimen\bibspace
\renewenvironment{thebibliography}[1]{%
 \section*{\refname 
       \@mkboth{\MakeUppercase\refname}{\MakeUppercase\refname}}%
     \list{\@biblabel{\@arabic\c@enumiv}}%
          {\settowidth\labelwidth{\@biblabel{#1}}%
           \leftmargin\labelwidth
           \advance\leftmargin\labelsep
           \itemsep\bibspace
           \parsep\z@skip     %
           \@openbib@code
           \usecounter{enumiv}%
           \let\p@enumiv\@empty
           \renewcommand\theenumiv{\@arabic\c@enumiv}}%
     \sloppy\clubpenalty4000\widowpenalty4000%
     \sfcode`\.\@m}
    {\def\@noitemerr
      {\@latex@warning{Empty `thebibliography' environment}}%
     \endlist}

\makeatother

\makeatletter

\newtheorem{thm}{Theorem}[section]
\newtheorem{lem}[thm]{Lemma}
\newtheorem{prop}[thm]{Proposition}

\newtheorem{cor}[thm]{Corollary}

\def\XXint#1#2#3{{\setbox0=\hbox{$#1{#2#3}{\int}$}
  \vcenter{\hbox{$#2#3$}}\kern-.5\wd0}}

\newcommand{\al}{\alpha}                \newcommand{\lda}{\lambda}
\newcommand{\om}{\Omega}                \newcommand{\pa}{\partial}
\newcommand{\va}{\varepsilon}           \newcommand{\ud}{\mathrm{d}}
\newcommand{\be}{\begin{equation}}      \newcommand{\ee}{\end{equation}}

\newcommand{\R}{\mathbb{R}}              \newcommand{\Sn}{\mathbb{S}^n}

\begin{document}

\title{\textbf{Bubbling and extinction for some fast diffusion equations in bounded domains}
\bigskip}

\author{\medskip  Tianling Jin\footnote{T. Jin is partially supported by Hong Kong RGC grants GRF 16302217, GRF 16306320 and NSFC 12122120.}, \  \
Jingang Xiong\footnote{J. Xiong was is partially supported by the National Key R\&D Program of China No. 2020YFA0712900, and NSFC grants 11922104.}}

\date{\today}

\maketitle

\begin{abstract}
We study a Sobolev critical fast diffusion equation in bounded domains with the Brezis-Nirenberg effect. We obtain extinction profiles of its positive solutions, and show that the convergence rates of the relative error in regular norms are at least polynomial. Exponential decay rates are proved for generic domains. Our proof makes use of its regularity estimates, a curvature type evolution equation, as well as blow up analysis. Results for Sobolev subcritical fast diffusion equations are also obtained.

\medskip 


\end{abstract}


\section{Introduction}

Let $\om$ be a bounded domain in $\R^n$, $n\ge 3$, with smooth boundary $\pa \om$. We consider the Cauchy-Dirichlet  problem for the  fast diffusion equation with the  Sobolev critical exponent
\be \label{eq:BN-1}
\begin{split}
\frac{\pa }{\pa t}u^{\frac{n+2}{n-2}} &= \Delta u+ b u\quad  \mbox{in }\om\times (0,\infty),\\
u&=0 \quad\quad\quad\quad \mbox{on }\pa\om \times (0,\infty),\\
u(\cdot,0)&=u_0\ge 0\quad\ \ \   \mbox{in }\om,
\end{split}
\ee
where $\Delta=\sum_{i=1}^n \frac{\pa^2 }{\pa x_i^2}$ is the Laplace operator, $u_0$ is not identically zero, and
\be\label{eq:b}
b\in[0,\lda_1)\mbox{ is a constant}
\ee
with $\lda_1$ being the first eigenvalue of $-\Delta$ in $\Omega$ with zero Dirichlet boundary condition. Hence, the operator $-\Delta -b$ is coercive on the Sobolev space $H_0^1(\om)$. The fast diffusion equations arise in the modelling of gas-kinetics, plasmas, thin liquid film dynamics driven by Van der Waals forces, and etc. If $b=0$, this Sobolev critical equation \eqref{eq:BN-1} can be viewed a unnormalized Yamabe flow with metrics degenerate on the boundary. 

The theory of existence and uniqueness of solutions to \eqref{eq:BN-1} is well understood, see V\'azquez \cite{Vaz,Vaz1}. If $u_0\in L^q(\om)$ for some $q>\frac{2n}{n-2}$, then the solution will become instantaneously positive in $\Omega$ and globally bounded. Moreover, the solution  will vanish in a finite time $T^*>0$. If we assume that $u_0\in H^1_0(\om)\cap L^q(\om)$ for some $q>\frac{2n}{n-2}$, then it follows from the work of Chen-DiBenedetto \cite{CDi}, DiBenedetto-Kwong-Vespri  \cite{DKV} and Jin-Xiong \cite{JX19} that the solution is of $C^{3,2}_{x,t}(\overline\om\times(0,T^*))$. In particular, the solutions are classical. Therefore, when we investigate the asymptotic behavior of nonnegative  solutions  to \eqref{eq:BN-1} as $t$ approaching to the extinction time $T^*$, there is no loss of generality to  consider classical (up to the boundary) solutions  to \eqref{eq:BN-1}.

When $\frac{n+2}{n-2}$ is replaced by $p\in (1,\frac{n+2}{n-2})$ if $n\ge 3$, or $p\in(1,\infty)$ if $n=1,2$, which is a Sobolev subcritical exponent, the extinction behavior of solutions to the fast diffusion equation
\be \label{eq:subcritical}
\begin{split}
\frac{\pa }{\pa t}u^{p} &= \Delta u\quad  \mbox{in }\om\times (0,\infty),\\
u&=0 \quad\quad \mbox{on }\pa\om \times (0,\infty)
\end{split}
\ee
has been well-studied.  By the scaling  
\be\label{eq:subchanging_variables}
v(x,t)=\Big (\frac{p}{(p-1)(T^*-\tau)}\Big)^{\frac{1}{p-1}} u(x,\tau), \quad t =\frac{p}{p-1}   \ln \Big(\frac{T^*}{T^*-\tau}\Big),
\ee
where $T^*$ is the extinction time, the equation \eqref{eq:subcritical} becomes
 \be \label{eq:subcritical2}
 \begin{split}
\frac{\pa }{\pa t}v^{p} &= \Delta v+v^{p} \quad \mbox{in }\om\times (0,\infty),\\
 v&=0 \quad \mbox{on }\pa\om \times (0,\infty). 
\end{split}
\ee  
Berryman-Holland \cite{BH} proved that the solution of \eqref{eq:subcritical2} converges to a stationary solution $v_\infty$ in $H^1_0(\Omega)$ along a sequence of times. Feireisl-Simondon \cite{FS} proved the full convergence in the $C^0(\overline\Omega)$ topology. Bonforte-Grillo-V\'azquez \cite{BGV}  proved that the relative error $v(\cdot,t)/v_\infty$ converges to $1$ in $L^\infty(\Omega)$.
Recently, Bonforte-Figalli \cite{BFig} proved the sharp exponential convergence of the relative error for generic domains $\Omega$, which means that the domains $\Omega$ satisfy
\be\label{condition2}
\begin{split}
&\mbox{For every nonnegative } H^1_0\ \mbox{ solution } v \mbox{ of }  -\Delta v- v^{p}=0 \mbox{ in } \Omega, \mbox{the linearized} \\[-1mm]
&\mbox{operator at }v, 
\mbox{that is } L_v:=-\Delta - p v^{p-1}, \mbox{ has a trivial kernel in } H^1_0(\Omega).  
\end{split}
\ee   
See Akagi \cite{Akagi} for another proof. The set of  smooth domains satisfying \eqref{condition2}  has generic properties, see Saut-Temam \cite{ST}. 

The main advantage of the subcritical regime is the upper bound of solutions $u$ to \eqref{eq:subcritical} proved in DiBenedetto-Kwong-Vespri  \cite{DKV}
\be \label{eq:idealbound}
u(x,t) \le C d(x)(T^*-t)^{\frac{1}{p-1}} \quad \mbox{for }t<T^*,
\ee
where $d(x)=dist(x,\pa \om)$. The estimate \eqref{eq:idealbound} implies that the function $v$ defined by \eqref{eq:subchanging_variables}, which satisfies \eqref{eq:subcritical2},  is uniformly bounded as $t\to\infty$, and consequently, has uniform regularity estimates up to the boundary $\partial\Omega$ by the work of \cite{CDi,DKV,JX19}. 

However, this uniform bound in general does not hold for \eqref{eq:subcritical2} if $p= \frac{n+2}{n-2}$. For instance, it is the case if $\om$ is star-shaped, since there is no stationary solution of \eqref{eq:subcritical2} due to the Pohozaev identity. In this paper, we will show that the uniform boundedness still holds for the equation \eqref{eq:BN-1} assuming $b>0$ and $n\ge 4$. The role of the positivity of $b$ when $n\ge 4$ was first discovered  in the seminal paper Brezis-Nirenberg \cite{BN}, and is similar to the role that the non-vanishing Weyl tensor and the positive mass theorem play in the resolution of the Yamabe problem on  compact manifolds by Aubin \cite{Aubin} and Schoen \cite{Schoen}.  

Under the scaling
 \be \label{eq:changing_variables}
 v(x,t)=\left(\frac{n+2}{4(T^*-\tau)}\right)^{\frac{n-2}{4}} u(x,\tau ), \quad t=\frac{n+2}{4} \ln \left(\frac{T^*}{T^*-\tau}\right),  
 \ee 
the equation \eqref{eq:BN-1} becomes
 \be \label{eq:BN-3}
 \begin{split}
\frac{\pa }{\pa t}v^{\frac{n+2}{n-2}} &= \Delta v+ b v +v^{\frac{n+2}{n-2}} \quad \mbox{in }\om\times (0,\infty),\\
 v&=0 \quad \mbox{on }\pa\om \times (0,\infty). 
\end{split}
\ee
We will show that every solution of \eqref{eq:BN-3}  converges to a stationary solution, that is a solution of
\be\label{eq:BNstationary}
\Delta v+ b v +v^{\frac{n+2}{n-2}}=0 \quad \mbox{in }\om,\quad v=0\quad \mbox{on }\partial\om,
\ee 
with at least polynomial rates. Moreover, the convergence rate will be  exponential if the domain $\Omega$ satisfies the following condition:
\be\label{condition}
\begin{split}
&\mbox{For every nonnegative } H^1_0\ \mbox{ solution } v \mbox{ of }  -\Delta v- b v -v^{\frac{n+2}{n-2}}=0 \mbox{ in } \Omega, \mbox{the linearized } \\[-1mm]
&\mbox{operator at }v, \mbox{that is } L_v:=-\Delta - b  -\frac{n+2}{n-2}v^{\frac{4}{n-2}}, \mbox{ has a trivial kernel in } H^1_0(\Omega).  
\end{split}
\ee   
The set of  smooth domains satisfying \eqref{condition} also has generic properties, see Saut-Temam \cite{ST}.

\begin{thm}\label{thm:critical}
Let $n\ge 4$, and $b>0$ satisfy \eqref{eq:b}. Let $u$ be a classical nonnegative solution of \eqref{eq:BN-1} with extinction time $T^*>0$. Let $v$ be defined by \eqref{eq:changing_variables}. Then there is a nonzero stationary solution $v_\infty$ of \eqref{eq:BN-3}, and two positive constants $\theta$ and $C$ such that 
\[
\left\|\frac{v(\cdot,t)}{v_\infty}-1\right\|_{C^2(\overline\om)}\le Ct^{-\theta}\quad\mbox{for all }t\ge 1.
\]
If $\Omega$ satisfies \eqref{condition}, then there exist  two positive constants $\gamma$ and $C$ such that 
\[
\left\|\frac{v(\cdot,t)}{v_\infty}-1\right\|_{C^2(\overline\om)}\le Ce^{-\gamma t}\quad\mbox{for all }t\ge 1.
\]
All the constants $\theta,\gamma$ and $C$ depend only on $n,b,\Omega$ and $u_0$.
\end{thm}

When $n=3$, it was shown in Brezis-Nirenberg \cite{BN} that the situation for the stationary equation \eqref{eq:BNstationary} changes drastically from dimensions $n\ge 4$. The positivity of $b$ is not sufficient to give a minimal energy solution of \eqref{eq:BNstationary}. Druet \cite{Druet} showed that the necessary and sufficient condition is the positivity of the regular part of the Green's function of $-\Delta-b$ at a diagonal point. There should be similar changes for the parabolic equation \eqref{eq:BN-3} as well. 

When $b=0$,  Sire-Wei-Zheng \cite{SWZ}  recently proved the existence of some initial data such that the solution of \eqref{eq:BN-3} blows up at finitely many points with an explicit blow up rate as  $t\to\infty$, using the gluing method for parabolic equations in the spirit of Cort\'azar-del Pino-Musso \cite{CdM} and D\'avila-del Pino-Wei \cite{DdW}.  This generalizes and provides rigorous proof of a result of Galaktionov-King \cite{GKing} for the radially symmetric case, where the solution blows up at one point. A class of type II ancient solutions to the Yamabe flow, which are rotationally
symmetric and converge to a tower of spheres as $t\to-\infty$, was constructed by  Daskalopoulos-del Pino-Sesum \cite{DdS}.  Bubble tower solutions for the energy critical heat equation were constructed in  del Pino-Musso-Wei \cite{dMW}.  It is conjectured in Sire-Wei-Zheng \cite{SWZ} that bubble tower solutions to \eqref{eq:BN-3} with $b=0$ also exist. Nevertheless, if it is the global case ($\Omega$ replaced by $\R^n$), then it has been proved by del Pino-S\'aez \cite{delSaez} that the solution of \eqref{eq:BN-3} for $b=0$ with fast decay initial data will converge to a nontrivial stationary solution, which is in fact a standard bubble.

To prove Theorem \ref{thm:critical},  we will adapt the blow up analysis of Struwe \cite{St}, Bahri-Coron \cite{Bahri-C}, Schwetlick-Struwe \cite{SS} and Brendle \cite{Br05}.  See also Chen-Xu \cite{CX} and Mayer \cite{Mayer} for similar analysis of scalar curvature flows. Here, we define a curvature type quantity $\mathcal{R}$, and derive its equation along the parabolic equation \eqref{eq:BN-3}. Due to the lack of information of  $\mathcal{R}$ on the boundary $\partial\Omega$, extra work is needed to obtain estimates for $\mathcal{R}$. Here the optimal boundary regularity proved in our previous paper \cite{JX19} is crucial.   Part  of the blow up  analysis in this paper remains valid  when $b=0$ or $n=3$.  The condition  $n\ge 4$ and $b>0$ is used in the final step (i.e., Corollaries \ref{cor:pure-singular} and \ref{cor:regular+singular}) to rule out bubbles, which is in the same spirit of Brezis-Nirenberg \cite{BN} in obtaining compactness of minimizing sequences.

Our proof of the polynomial decay rates in Theorem \ref{thm:critical} can be applied to prove the polynomial rate of the convergence of the relative error for the Sobolev subcritical fast diffusion equation \eqref{eq:subcritical2} in all smooth domains. We also provide an alternative proof the exponential convergence result of Bonforte-Figalli \cite{BFig} for $\Omega$ satisfying \eqref{condition2}.
\begin{thm}\label{thm:subcritical}
Suppose $p\in (1,\frac{n+2}{n-2})$ if $n\ge 3$, and $p\in(1,\infty)$ if $n=1,2$. Let $u$ be a classical nonnegative solution of \eqref{eq:subcritical} with extinction time $T^*>0$. Let $v$ be defined by \eqref{eq:subchanging_variables}. Then there is a stationary solution $v_\infty$ of \eqref{eq:subcritical2}, and two positive constants $\theta$ and $C$ such that 
\[
\left\|\frac{v(\cdot,t)}{v_\infty}-1\right\|_{C^2(\overline\om)}\le Ct^{-\theta}\quad\mbox{for all }t\ge 1.
\]
If $\Omega$ satisfies \eqref{condition2}, then there exist  two positive constants $\gamma$ and $C$ such that 
\[
\left\|\frac{v(\cdot,t)}{v_\infty}-1\right\|_{C^2(\overline\om)}\le Ce^{-\gamma t}\quad\mbox{for all }t\ge 1.
\]
All the constants $\theta,\gamma$ and $C$ depend only on $n,p,\Omega$ and $u_0$.
\end{thm} 

All the rates $\theta$ and $\gamma$ in both Theorem \ref{thm:critical} and Theorem \ref{thm:subcritical} are not explicit. As mentioned earlier, the sharp exponential convergence rate of the relative error for \eqref{eq:subcritical2}  under the condition \eqref{condition2} was  obtained by Bonforte-Figalli \cite{BFig} (see also Akagi \cite{Akagi} for a different method).  We do not pursuit the sharpness of $\gamma$ in this paper.

 We know from the work of Carlotto-Chodosh-Rubinstein \cite{CCR}  that  there exists a Yamabe flow on $\mathbb{S}^1(1/\sqrt{n-2})\times\mathbb{S}^{n-1}(1)$ such that it converges exactly at a polynomial rate. Recently, Choi-McCann-Seis \cite{CMS} proved that  for the solutions of the fast diffusion equation \eqref{eq:subcritical2}, the relative error either decays exponentially with the sharp rate or else decays algebraically at a rate $1/t$ or slower. 

This paper is organized as follows. Sections \ref{sec:integralbound}--\ref{sec:convergence} deal with the critical equation \eqref{eq:BN-3}. We first obtain certain integral bounds for solutions of this critical equation in Section \ref{sec:integralbound}. Sections \ref{sec:concentration} is for the possible concentration phenomenon for its solutions. In Section \ref{sec:blowup}, we use blow up analysis to rule out such possible concentration phenomenon. Section \ref{sec:convergence} is devoted to the proof of the uniform boundedness and convergence results in Theorem \ref{thm:critical}. In Section \ref{sec:subcritical}, we consider  the subcritical equation \eqref{eq:subcritical2} and prove Theorem \ref{thm:subcritical}.

\bigskip

\noindent \textbf{Acknowledgement:} Part of this work was completed while the second named author was visiting the Hong Kong University of Science and Technology and Rutgers University, to which he is grateful for providing  very stimulating research environments and supports. Both authors would like to thank Professors Ha\"im Brezis and YanYan Li for their interests and useful comments. 
 
\section{Integral bounds}\label{sec:integralbound}  
 
 For an open set $\om$, let $H_0^1(\om)$ be the closure of $C_c^\infty(\om)$ under the norm 
\[
\|u\|_{H_0^1(\om)}:=\left(\int_\om |\nabla u|^2 \,\ud x\right)^{1/2}.
\] 
For convenience, we define 
\be \label{eq:norm}
\|u\|:=\left(\int_\om (|\nabla u|^2-bu^2) \,\ud x\right)^{1/2}
\ee
and 
\be \label{eq:inner-product}
\langle u,v\rangle= \int_{\om} (\nabla u\nabla v-buv)\,\ud x
\ee
be the associated inner product.  Since 
\[
\left(1-\frac{b}{\lambda_1}\right)\int_\om |\nabla u|^2 \,\ud x\le \int_\om (|\nabla u|^2-bu^2) \,\ud x
\]
and we assumed $b<\lda_1$, by the Sobolev inequality, there exists a constant $K_b>0$ such that 
\be \label{eq:sobolev-b}
\|u\|_{L^{\frac{2n}{n-2}}(\Omega)} \le K_b^{1/2} \|u\| \quad \mbox{for any }u\in H_0^1(\om).
\ee 

Recall that from \cite{JX19}, we know that the solution $u(x,t)$ of \eqref{eq:BN-1} is smooth in $t\in (0,T^*)$ for every $x\in\overline\Omega$, and $ \pa_t^l u(\cdot,t) \in C^{\frac{3n-2}{n-2}}(\overline \om)$ for all $l\ge 0$ and all $t\in(0,T^*)$.  
 
\begin{lem}\label{lem:basic} Let $u$ be a solution of \eqref{eq:BN-1}, and $T^*$ be the extinction time of $u$. Then for every $0<t<T^*$,
 \[
\frac{1}{C}(T^*-t)^{\frac{n}{2}} \le  \int_\om u(x,t)^{\frac{2n}{n-2}}\,\ud x \le C(T^*-t)^{\frac{n}{2}},
 \]
 where $C$ is a positive constant depending only on $n,b,\om$ and $u_0$. 
 
 \end{lem}
 
 \begin{proof} If $b=0$, the lemma was proved by \cite{BH} (noting that by our regularity result in \cite{JX19}, the regularity assumptions in \cite{BH} are satisfied, and thus the calculations in [4] are justified).   The same proof applies if $b\in (0,\lambda_1)$ by using \eqref{eq:sobolev-b}. We sketch it in the below for reader's convenience.

Let
\[
\xi(t)=\left(\int_\om u(x,t)^{\frac{2n}{n-2}}\,\ud x\right)^\frac{2}{n}\quad\mbox{and}\quad S(t)=\frac{\|u(\cdot,t)\|^2}{\|u(\cdot,t)\|_{L^{\frac{2n}{n-2}}(\Omega)}^2}.
\]
Then
\[
\frac{\ud}{\ud t}\xi(t)=-\frac{4}{n+2}S(t)\le -\frac{4}{n+2} K_b^{-1},
\]
where we used \eqref{eq:sobolev-b} in the last inequality. Then the first inequality of this lemma follows by integrating the above inequality from $t$ to $T^*$. 

Making use of the equation \eqref{eq:BN-1} and the same arguments in \cite{BH},  we have
\[
\frac{\ud}{\ud t} \|u(\cdot,t)\|^2\le 0 \quad\mbox{and}\quad \frac{\ud}{\ud t} S(t)\le 0.
\]
Hence, both $\|u(\cdot,t)\|$ and $S(t)$ are non-increasing in $t$. Since
\[
\frac{\ud}{\ud t} \int_\om u(x,t)^{\frac{2n}{n-2}}\,\ud x=-\frac{2n}{n+2}\|u(\cdot,t)\|^2,
\]
then by integrating the above inequality from $t$ to $T^*$ and using the monotonicity of $\|u(\cdot,t)\|$ and $S(t)$, we have
\[
\int_\om u(x,t)^{\frac{2n}{n-2}}\,\ud x\le \frac{2n}{n+2} (T^*-t)\|u(\cdot,t)\|^2\le\frac{2n}{n+2} (T^*-t)S(0) \|u(\cdot,t)\|_{L^{\frac{2n}{n-2}}(\Omega)}^2.
\]
This leads to the second inequality of this lemma.
 \end{proof}
 
 Let $v$ be as in \eqref{eq:changing_variables}. By Lemma \ref{lem:basic}, we have 
\be \label{eq:BN-5} 
\frac{1}{C} \le \int_{\om} v(x,t)^{\frac{2n}{n-2}}\,\ud x\le C
\ee
for all $t>0$, where $C$ is a positive constant depending only on $n,b,\Omega$ and $u_0$.  Define 
\be \label{eq:F-funct}
F(v(t))= \int_{\om } \Big( |\nabla v(x,t)|^2 -b v(x,t)^2 -\frac{n-2}{n} v(x,t)^{\frac{2n}{n-2}}\Big)\,\ud x. 
\ee
It follows that $F(v(t))$ is bounded from below for all $t>0$. By the equation of $v$ and integrating by parts, 
\be \label{eq:F-dcr}
\frac{\ud}{\ud t} F(v(t))= - 2\int_{\om}(\Delta v +bv+v^{\frac{n+2}{n-2}}) \pa_t v \,\ud x= -\frac{2(n+2)}{n-2} \int_{\om }v^{\frac{4}{n-2}} |\pa_t v|^2 \,\ud x \le 0. 
\ee
Hence, $F(v(t))$ is non-increasing in $t$, and thus, together with \eqref{eq:BN-5}, we have $\|v(\cdot,t)\|_{H^1_0(\Omega)}$ is uniformly bounded. Moreover, there exists some constant $F_\infty$ such that 
\be \label{eq:f_infty}
\lim_{t\to \infty} F(v(t)) = F_\infty. 
\ee

Define 
\begin{equation}\label{eq:Q}
\mathcal{R}= v^{-\frac{n+2}{n-2}}(-\Delta v-bv)
\end{equation}
and 
\be\label{eq:definitionMq}
M_q(t)= \int_{\om} |\mathcal{R} -1|^q v^{\frac{2n}{n-2}} \,\ud x, \quad  q \ge 1. 
\ee 
In \cite{JX19}, we proved that $\mathcal{R}=1-\frac{n+2}{n-2} \frac{\pa_t v}{v}$ is $C^{2}$ up to the boundary $\partial\Omega$. However, all the estimates there for solutions of \eqref{eq:BN-3} are only locally uniform in $t\in (0,\infty)$. We shall prove some uniform estimates for all $t\in [1,\infty)$ and $M_q(t) \to 0$ as $t\to \infty$. 

To do this, we will first use Moser's iteration to obtain a uniform lower bound of $\mathcal{R}$ as an intermediate step. So we need the following evolution equation of $\mathcal{R}$ and integration by parts formula. 
\begin{lem} \label{lem:flow-properties} Let $g=v^{\frac{4}{n-2}} g_{flat}$. Then
\begin{itemize}
\item[(i).]  
\be\label{eq:volumn}
\pa_t v^{\frac{2n}{n-2}}= -\frac{2n}{n+2}(\mathcal{R}-1) v^{\frac{2n}{n-2}}.
\ee
\item[(ii).] 
\be\label{eq:R-1}
\pa_t (\mathcal{R}-1)=\frac{n-2}{n+2} \Delta_g (\mathcal{R}-1) +\frac{4}{n+2} (\mathcal{R}-1)^2 + \frac{4}{n+2} (\mathcal{R}-1),
\ee
where $\Delta_g $ is the Laplace-Beltrami operator of $g$.
\item[(iii).] For any $f\in H^2(\om)$ and $h\in  H^1(\om)$,  
\be\label{eq:integrationbyparts}
\int_{\om}h \Delta_g f\,\ud vol_g = -\int_{\om} \langle \nabla_g f, \nabla_g h\rangle_g \,\ud vol_g. 
\ee
\end{itemize}

\end{lem} 

\begin{proof} The equation \eqref{eq:volumn} follows immediately from \eqref{eq:BN-3} and \eqref{eq:Q}. We also have $\pa_t v= \frac{n-2}{n+2} v (1-\mathcal{R})$.  

By the definition of $\mathcal{R}$, we have
\begin{align*}
\pa_t (\mathcal{R}-1)&= \frac{n+2}{n-2} v^{-\frac{2n}{n-2}} \pa_t v (\Delta+ b) v- v^{-\frac{n+2}{n-2}}  (\Delta +b) \pa_t v\\&
= v^{-\frac{n+2}{n-2}} (1-\mathcal{R}) (\Delta+ b) v  - \frac{n-2}{n+2} v^{-\frac{n+2}{n-2}}  (\Delta +b) (v(1-\mathcal{R}) )\\&
=(\mathcal{R}-1)\mathcal{R} - \frac{n-2}{n+2} v^{-\frac{n+2}{n-2}}  (\Delta +b) (v(1-\mathcal{R}) ). 
\end{align*}
Let $L_g= \Delta _g -\frac{n-2}{4(n-1)}R_g$ be the conformal Laplacian of $g$, where $\Delta _g$ is the  Laplace--Beltrami operator of the metric $g$ and $R_g$ is the the scalar curvature of $g$. By the conformal transformation law
\[
L_g (v^{-1}\varphi)=v^{-\frac{n+2}{n-2}}\Delta\varphi,\quad\forall \varphi\in C^2(\Omega),
\] 
we have 
\[
\frac{n-2}{4(n-1)}R_g=-L_g(1)=-v^{-\frac{n+2}{n-2}}\Delta v=\mathcal{R} +b v^{-\frac{4}{n-2}}
\]
and 
\begin{align*}
v^{-\frac{n+2}{n-2}}  (\Delta +b) (v(1-\mathcal{R}) )&=  L_g (1-\mathcal{R}) + bv^{-\frac{4}{n-2}} (1-\mathcal{R})\\&
= \Delta _g (1-\mathcal{R}) -\frac{n-2}{4(n-1)} R_g (1-\mathcal{R}) + bv^{-\frac{4}{n-2}} (1-\mathcal{R}) \\&
=  \Delta _g (1-\mathcal{R}) -\mathcal{R} (1-\mathcal{R}).
\end{align*}
Then, \eqref{eq:R-1} follows.  
 
Finally,  
\begin{align*}
\int_{\om}h \Delta_g f\,\ud vol_g &= \int_\Omega h  v^{-\frac{2n}{n-2}} \pa_i (v^{\frac{2n}{n-2}}  v^{-\frac{4}{n-2}} \pa_i  f) v^{\frac{2n}{n-2}}\,\ud x\\& 
=  \int_\Omega h \pa_i (v^2 \pa_if) \,\ud x= - \int_\Omega v^2 \pa_i f \pa_i h\,\ud x= -\int_\Omega \langle \nabla_g f, \nabla_g h\rangle_g \,\ud vol_g,
\end{align*}
where we used $v=0$ on $\partial\Omega$ in the third equality.
\end{proof}

We have the following Sobolev inequality regarding the metric $g=v^{\frac{4}{n-2}} g_{flat}$:
\begin{lem} \label{lem:sob-elliptic}
There holds 
\[
\left(\int_{\om} |f|^{\frac{2n}{n-2}}\,\ud vol_g \right)^{\frac{n-2}{n}}\le K_b\int_{\om}( |\nabla_g f|_g^2+\mathcal{R}f^2) \,\ud vol_g
\]
for any $f\in H^1(\overline \om)$, where $K_b$ is the constant in \eqref{eq:sobolev-b}.
\end{lem}
\begin{proof} 
Note that 
\[
|\nabla (fv)|^2= v^2 |\nabla f|^2 +f^2|\nabla v|^2 +2 vf\nabla v \cdot \nabla f, 
\]
\begin{align*}
\int_\Omega (f^2|\nabla v|^2 +2 v\nabla v f \nabla f)\,\ud x&=  \int_\Omega (f^2|\nabla v|^2 +v\nabla v  \nabla f^2 ) \,\ud x\\&=
-\int_\Omega vf^2\Delta v \,\ud x\\&
= \int_\Omega (\mathcal{R} f^2 v^{\frac{2n}{n-2}}+ b v^2 f^2) \,\ud x. 
\end{align*}
Hence, 
\[
\int_{\om}( |\nabla_g f|_g^2+\mathcal{R}f^2) \,\ud vol_g=\int_{\om}( v^2|\nabla f|^2+\mathcal{R}f^2v^{\frac{2n}{n-2}}) \,\ud x = \int_\om (|\nabla (fv)|^2- b(fv)^2 )\,\ud x. 
\]
Therefore, the lemma follows from \eqref{eq:sobolev-b}. 
\end{proof} 

For any $t_0\ge 0$ and $T> 0$, let 
\[
V^1(\om \times (t_0,t_0+T)) = C^0((t_0,t_0+T); L^2(\om)) \cap L^2 ((t_0,t_0+T); H^1(\om)), 
\]
equipped with the norm 
\[
\|f\|_{V^1(\om \times (t_0,t_0+T))}^2= \sup_{t_0<t<t_0+T} \int_{\om } f(x,t)^2 \,\ud vol_g +\int_{t_0}^{t_0+T} \int_{\om} ( |\nabla_g f|_g^2+\mathcal{R}f^2) \,\ud vol_g \ud t. 
\]
 
We have the following parabolic version of Sobolev inequality.
\begin{lem}\label{lem:sob-parabolic} For any $f\in V^1 (\om \times (t_0,t_0+T))$, we have
\[
\left(\int_{t_0}^{t_0+T} \int_{\om} |f|^{\frac{2(n+2)}{n}}\,\ud vol_g \ud t\right)^{\frac{n}{n+2}}\le K_b^{\frac{n}{n+2}} \|f\|_{V^1(\om \times (t_0,t_0+T))}^2. 
\]
\end{lem}

\begin{proof} By H\"older's inequality and Lemma \ref{lem:sob-elliptic}, we have 
\begin{align*}
 \int_{\om} |f|^{\frac{2(n+2)}{n}}\,\ud vol_g & =  \int_{\om} |f|^2|f|^{\frac{4}{n}}\,\ud vol_g \\&
 \le \Big( \int_{\om} |f|^{\frac{2n}{n-2}}\,\ud vol_g\Big )^{\frac{n-2}{n}} \Big( \int_{\om} |f|^{2}\,\ud vol_g\Big)^{\frac{2}{n}}\\&
 \le K_b\int_{\om}( |\nabla_g f|_g^2+\mathcal{R}f^2) \,\ud vol_g \Big( \int_{\om} |f|^{2}\,\ud vol_g\Big)^{\frac{2}{n}}. 
\end{align*}
Hence, by  Young's inequality 
\begin{align*}
&\left(\int_{t_0}^{t_0+T} \int_{\om} |f|^{\frac{2(n+2)}{n}}\,\ud vol_g \ud t\right)^{\frac{n}{n+2}} \\&\le K_b^{\frac{n}{n+2}} \left(\int_{\om}( |\nabla_g f|_g^2+\mathcal{R}f^2) \,\ud vol_g \right)^{\frac{n}{n+2}} \left( \sup_{t_0<t<t_0+T} \int_{\om } f(x,t)^2 \,\ud vol_g \right)^{\frac{2}{n+2}}\\&
\le K_b^{\frac{n}{n+2}}  \|f\|_{V^1(\om \times (t_0,t_0+T))}^2. 
\end{align*}
Therefore, the proof is completed. 
\end{proof}

With the Sobolev inequality in Lemma \ref{lem:sob-parabolic}, we will apply Moser's iterations to the equation \eqref{eq:R-1} to obtain a uniform lower bound of $\mathcal{R}$.
\begin{lem}\label{lem:Rlowerbound} For $t\ge 1$, we have 
\[
\mathcal{R}-1\ge -C,
\]
where $C$ is a constant depending only on $\om,n,b$ and $v_0$.  

\end{lem} 

\begin{proof} 
 Let $T>2$, $\frac 12\le T_2<T_1\le1$, $\eta(t)$ be a smooth cut-off function so that $\eta(t)=0$ for all $t<T_2$, $0\le \eta(t)\le 1$ for $t\in [T_2,T_1]$, $\eta(t)=1$ for all $t>T_1$, and $|\eta'(t)|\le \frac{2}{T_1-T_2}$. Denote $\phi= (1-\mathcal{R})^+$. 
By \eqref{eq:R-1}, we have 
\[
\pa_t (1-\mathcal{R})= \frac{n-2}{n+2} \Delta_g (1-\mathcal{R})-\frac{4}{n+2} (1-\mathcal{R})^2 +\frac{4}{n+2} (1-\mathcal{R}).
\] 
Let $k\ge \frac{n}{2}-1$ be a real number. Multiplying both sides of the inequality by $\eta^2 \phi^{1+k}$ and integrating by parts, we see that, for any $0< s<T$, 
\begin{align*}
&\frac{1}{2+k}\int_0^s  \int_\om \eta^2 \pa_t \phi^{2+k}  \,\ud vol_g\ud t +  \frac{4(n-2)(k+1)}{(n+2)(k+2)^2} \int_0^s \int_\om \eta^2 |\nabla_g \phi^{\frac{k+2}{2}}|_g^2 \,\ud vol_g\ud t  \\&
\le -\frac{4}{n+2}\int_0^s \int_\om \phi^{3+k} \eta ^2 \,\ud vol_g\ud t  +\frac{4}{n+2} \int_0^s \int_\om \phi^{2+k} \eta ^2 \,\ud vol_g\ud t. 
\end{align*}
Note that using \eqref{eq:volumn},  we have
\begin{align*}
&\frac{1}{2+k}\int_0^s \int_\om \eta^2 \pa_t \phi^{2+k} \,\ud vol_g\ud t  \\
&= \frac{1}{2+k} \int_\om \phi^{2+k} \eta^2 \,\ud vol_g \Big|_{t=s} - \frac{1}{2+k} \int_0^s \int_\om \phi^{2+k} \left(2\eta\pa_t \eta +\frac{2n}{n+2} (1-\mathcal{R})\eta^2\right)  \,\ud vol_g\ud t\\
&= \frac{1}{2+k} \int_\om \phi^{2+k} \eta^2 \,\ud vol_g \Big|_{t=s} - \frac{1}{2+k} \int_0^s \int_\om  \left(2\phi^{2+k}\eta\pa_t \eta +\frac{2n}{n+2} \phi^{3+k}\eta^2\right)  \,\ud vol_g\ud t. 
\end{align*}
Note that the term $\frac{2n}{n+2} (1-\mathcal{R})\eta^2$ in the above comes from the derivative of the volume form $\ud vol_g$ in $t$. Since  $k\ge \frac{n}{2}-1$, $\frac{1}{2+k}\frac{2n}{n+2}< \frac{4}{n+2}$. Furthermore, 
\[
\int_{\om} \mathcal{R} \phi^{2+k}\eta^2 \,\ud vol_g= - \int_{\om} (1-\mathcal{R}) \phi^{2+k}\eta^2\,\ud vol_g + \int_\om \phi^{2+k}\eta^2 \,\ud vol_g\le  \int_\om \phi^{2+k}\eta^2\,\ud vol_g. 
\] It follows that  
\[
\|\eta \phi^{\frac{2+k}{2}}\|_{V^1(\om\times (0,T))}^2 \le C(2+k)  \int_0^T \int_\om \phi^{2+k} (\eta ^2+|\pa_t \eta| \eta) \,\ud vol_g\ud t,
\]
where $C>0$ depends only on $n$. Making use of Lemma  \ref{lem:sob-parabolic}, we have for all $\gamma:=k+2\ge \frac{n+2}{2}$ that
\[
\left(\int_{T_1}^T \int_\om \phi^{\frac{\gamma(n+2)}{n}} \,\ud vol_g\ud t\right)^{\frac{n}{\gamma(n+2)}} \le \left(\frac{C\gamma}{T_1-T_2}\right)^{\frac{1}{\gamma}}  \left(\int_{T_2}^T \int_\om \phi^{\gamma} \,\ud vol_g\ud t\right)^{\frac{1}{\gamma}}  ,
\]
By the standard Moser's iteration argument,  we have 
\[
\sup_{\om\times [1,T]} \phi \le C(n,K_b) \Big(\int_{1/2}^T\int_\om \phi^{\frac{n+2}{2}}\,\ud vol_g\ud t \Big)^{\frac{2}{n+2}},
\]
where $C(n,K_b)>0$ depending only on $n$ and $K_b$. Thus,
\begin{align*}
\sup_{\om\times [1,T]} \phi &\le C(n,K_b) \Big(\int_{1/2}^1\int_\om \phi^{\frac{n+2}{2}}\,\ud vol_g\ud t \Big)^{\frac{2}{n+2}}+C(n,K_b) \Big(\int_{1}^T\int_\om \phi^{\frac{n+2}{2}}\,\ud vol_g\ud t \Big)^{\frac{2}{n+2}}\\
&\le C(n,K_b) \|\mathcal{R}-1\|_{L^\infty (\om \times (1/2,1))}+C(n,K_b) \Big(\int_{1}^T\int_\om \phi^{\frac{n+2}{2}}\,\ud vol_g\ud t \Big)^{\frac{2}{n+2}},
\end{align*}
By Young's inequality, we have 
\begin{align*}
\Big(\int_1^T\int_\om \phi^{\frac{n+2}{2}}\,\ud vol_g\ud t \Big)^{\frac{2}{n+2}} &\le (\sup_{\om\times [1,T]} \phi)^{\frac{n-2}{n+2}} \Big(\int_1^T\int_\om \phi^{2}\,\ud vol_g\ud t \Big)^{\frac{2}{n+2}}\\&
\le \va \sup_{\om\times [1,T]} \phi +C(\va) \Big(\int_1^T\int_\om \phi^{2}\,\ud vol_g\ud t \Big)^{\frac{1}{2}},
\end{align*}
for any small constant $\va$. Therefore, by choosing a small $\va$, we have 
\be \label{eq:finite-T}
\sup_{\om\times [1,T]} \phi \le C(n,K_b)\left\{  \|\mathcal{R}-1\|_{L^\infty (\om \times (1/2,1))}+ \Big(\int_{1}^T M_2\,\ud t \Big)^{\frac{1}{2}} . \right\}
\ee 
By \eqref{eq:F-dcr} and the definition of $\mathcal{R}$, we have 
\[
\frac{\ud }{\ud t}F(v(t))= -\frac{2(n-2)}{n+2}M_2(t). 
\]
It follows that 
\be \label{eq:m-2}
\int_0^\infty M_2(t)\,\ud t \le \frac{n+2}{2(n-2)}(F(v(0))- F_\infty)<\infty.
\ee 
Moreover, it was proved in \cite{JX19} that   $\|\mathcal{R}-1\|_{L^\infty (\om \times (1/2,1))} \le C$.  Sending $T\to \infty$ in \eqref{eq:finite-T},  we have 
\[
\sup_{\om\times [1,\infty)} (1-\mathcal{R})^+ \le C.   
\]
Therefore, the proof is completed. 
\end{proof}  

Using this uniform lower bound of $\mathcal{R}$, we can derive some useful differential inequalities for $M_q$ defined in \eqref{eq:definitionMq}.

For $q>1$,  using Lemma \ref{lem:flow-properties}, 
 we have
\begin{align*}
\frac{\ud M_q}{\ud t}&= \int_\Omega q |\mathcal{R}-1|^{q-2} (\mathcal{R}-1 )\frac{\pa }{\pa t} (\mathcal{R}-1) \,\ud vol_g -\int_\Omega (\mathcal{R}-1)^{q} \frac{\partial}{\partial t}v^{\frac{2n}{n-2}}\,\ud x   \\&
=q\frac{n-2}{n+2} \int_\Omega |\mathcal{R}-1|^{q-2} ( \mathcal{R}-1) \Delta_g ( \mathcal{R}-1)\,\ud vol_g \\
&\quad + \frac{4q}{n+2}\int_\Omega   |\mathcal{R}-1|^{q} \,\ud vol_g  +\frac{4}{n+2}(q-\frac{n}{2}) \int_\Omega   |\mathcal{R}-1|^{q} ( \mathcal{R}-1)\,\ud vol_g.
\end{align*}
Using Lemma \ref{lem:Rlowerbound},  we have for $t\ge 1$ that
\[
\left|\int_\Omega   |\mathcal{R}-1|^{q} ( \mathcal{R}-1)\,\ud vol_g-\int_\Omega   |\mathcal{R}-1|^{q+1}\,\ud vol_g\right|= 2\int_\Omega   |\mathcal{R}-1|^{q} ( \mathcal{R}-1)^-\,\ud vol_g\le CM_q.
\]
Using Lemma \ref{lem:flow-properties}, we have
\begin{align*}
\int_\Omega |\mathcal{R}-1|^{q-2} ( \mathcal{R}-1) \Delta_g ( \mathcal{R}-1)\,\ud vol_g= -\frac{4(q-1)}{q^2}\int_\Omega  |\nabla_g |\mathcal{R}-1|^{\frac{q}{2}}|_g^2\,\ud vol_g\le 0.
\end{align*}
Therefore, for $q\le \frac{n}{2}$ we have,
\begin{equation}\label{eq:diff-ineq-2}
\frac{\ud M_q}{\ud t} + \frac{4}{n+2}\left(\frac{n}{2}-q\right) M_{q+1}\le CM_q\quad\mbox{for }t\ge 1,
\end{equation}
where $C>0$ is a constant depending on $q$.

For $q\ge \frac{n}{2}$, we first obtain from Lemma \ref{lem:sob-elliptic} that
\begin{align*}
&\int_\Omega |\mathcal{R}-1|^{q-2} ( \mathcal{R}-1) \Delta_g ( \mathcal{R}-1)\,\ud vol_g\\
&= -\frac{4(q-1)}{q^2}\int_\Omega  |\nabla_g |\mathcal{R}-1|^{\frac{q}{2}}|_g^2\,\ud vol_g\\
&\le -\beta M_{\frac{qn}{n-2}}^{\frac{n-2}{n}}  +\frac{4(q-1)}{q^2} \int_\Omega \mathcal{R} |\mathcal{R}-1|^q\,\ud vol_g\\
&\le -\beta M_{\frac{qn}{n-2}}^{\frac{n-2}{n}}  +\frac{4(q-1)}{q^2} \int_\Omega (\mathcal{R}-1) |\mathcal{R}-1|^q\,\ud vol_g +\frac{4(q-1)}{q^2} M_q,
\end{align*}
where $\beta>0$ is a constant depending on $K_b$ and $q$. Then, we have
\[
\frac{\ud M_q}{\ud t} +  \beta M_{\frac{qn}{n-2}}^{\frac{n-2}{n}} \le  \frac{4}{n+2}\left(q-\frac{n}{2}+\frac{(n-2)(q-1)}{q}\right) M_{q+1}+CM_q.
\]
By the interpolation inequality and Young's inequality we have
\begin{align*}
M_{q+1} \le M_{\frac{qn}{n-2}}^{\frac{n-2}{2q}} M_{q}^{\frac{2(q+1)-n}{2q}} \le \va M_{q(p+1)/2}^{\frac{n-2}{n}} +C(\va) M_{q}^{\frac{2(q+1)-n}{2q-n}}.
\end{align*}
By choosing a small $\va$, we obtain
\be \label{eq:diff-ineq-3}
\begin{split}
\frac{\ud }{\ud t} M_q(t)  +\beta M_{\frac{qn}{n-2}}(t)^{\frac{n-2}{n}} &\le C\Big(M_q(t)+  M_{q}(t)^{1+\frac{2}{2q-n}}\Big) \quad\mbox{for }t\ge 1
\end{split}
\ee 
for $q>\frac{n}{2}$, where $\beta$ and $C$ are positive constants depending on $q$.  

The differential inequalities \eqref{eq:diff-ineq-2} and \eqref{eq:diff-ineq-3} will be used recursively to prove the decay of $M_q$ for all $q\ge 1$. 

\begin{prop}\label{prop:Mq} For every $1\le q<\infty$, we have 
\[
\lim_{t\to \infty} M_q(t) =0. 
\]
\end{prop} 

\begin{proof} 
By H\"older's inequality and \eqref{eq:BN-5}, we only need to consider $q\ge 2$.

The idea of the proof will go recursively as follows.  Note that if the right hand sides of  \eqref{eq:diff-ineq-2} and \eqref{eq:diff-ineq-3}  are integrable in $[1,\infty)$, then by integrating both sides, and noticing that $ \frac{4}{n+2}\left(\frac{n}{2}-q\right) M_{q+1}$ with $q\le n/2$ and $\beta M_{\frac{qn}{n-2}}(t)^{\frac{n-2}{n}} $ are nonnegative and thus can be dropped,  we will have $M_q(t)\to 0$ as $t\to \infty$. Integrating again including these two nonnegative terms will in return show that they are integrable. This iteration shows the integrability and the limit of $M_{q+1}$ or $M_{\frac{qn}{n-2}}$ from $M_q$. The starting point of this iteration is $q=2$, because of \eqref{eq:m-2}. This gives us a desired sequence $\{q_k\}$ for which the proposition holds. The conclusion for all $q$ is then followed by H\"older's inequality and \eqref{eq:BN-5}. The details of the proof are given in the below.

Let us assume $n\ge 4$ first. 

Case 1. $2\le q\le \frac{n}{2}$. 

Since $M_2\in L^1(0,\infty)$, we can pick $t_j\to \infty$ such that $M_2(t_j) \to 0$ as $j\to \infty$.  By \eqref{eq:diff-ineq-2} we have 
\[
\frac{\ud }{\ud t} M_2(t) \le CM_2(t). 
\]
Integrating the above inequality we have 
\[
M_2(t)\le M_2(t_j) +C \int_{t_j}^\infty M_2(s)\,\ud s \quad \mbox{for }t\ge t_j. 
\]
Hence,  $\lim_{t\to \infty} M_2(t)=0$.  If $2<q\le \frac{n}{2}$, \eqref{eq:diff-ineq-2} we have 
\[
\int_{1}^\infty M_3(t)\,\ud t \le C \Big(\int_1^\infty M_2(t)\,\ud t +M_2(1)\Big) <\infty. 
\]
For any $2<q\le \min\{3,\frac{n}{2}\}$, we have $M_q(t)\le M_2(t)+M_3(t)$. Hence, $\int_{1}^\infty M_q(t)\,\ud t <\infty$. We can repeat the argument for $M_2$ to show that $M_q(t) \to 0$ as $t\to \infty$.  If $3<\frac{n}{2}$, we can show that $\int_1^\infty M_4<\infty$ and $M_{q}(t)\to 0$ as $t\to \infty$ for all $3< q\le  \min\{4,\frac{n}{2}\}$. Repeating this argument in finite times, and using H\"older's inequality with \eqref{eq:BN-5}, we then have $M_q \in L^1(1,\infty)$ and $M_{q}(t)\to 0$ as $t\to \infty$ for all $2\le q\le \frac{n}{2}$.

Case 2. $q> \max\{2, \frac{n}{2}\}$.

By \eqref{eq:diff-ineq-2} with $q=n/2$, we have  
\be\label{eq:q0}
\int_1^\infty M_{\frac{n^2}{2(n-2)}}(t)^{\frac{n-2}{n}}\,\ud t <\infty.
\ee
Using \eqref{eq:diff-ineq-3} to have 
\[
\frac{\ud }{\ud t} M_{q}(t) \le CM_{q}(t)^{\frac{n-2}{n}}\Big(M_{q}(t)^{\frac{2}{n}}+  M_{q}(t)^{\frac{2}{n}+\frac{2}{2q-n}}\Big) \quad\mbox{for }t\ge 1
\]
Hence,
\[
H(M_q(t))\le H(M_{q}(T)) +C\int_T^\infty M_{q}^{\frac{n-2}{n}}\,\ud t \quad \mbox{for } 1\le T<t<\infty
\]
where $$H(\rho)=\int_0^\rho \frac{1}{s^{\frac{2}{n}}+s^{\frac{2}{n}+\frac{2}{2q-n}}}\,\ud s.$$ 
Let 
\[
q_0=\frac{n^2}{2(n-2)},\quad q_k=\frac{n}{n-2}q_{k-1}, \quad k=1,2,\cdots.
\]  
Note that $\frac{2}{n}+\frac{2}{2q_0-n}=1$ and $\frac{2}{n}+\frac{2}{2q_k-n}<1$ for all $k\ge 1$. Hence, starting with \eqref{eq:q0} that $\displaystyle\int_1^\infty M_{q_{0}}(t)^{\frac{n-2}{n}}\,\ud t <\infty$, using similar arguments to those in Case 1, we can recursively prove in the order of $k=0,1,2,\cdots$ that $M_{q_k}(t)\to 0$ as $t\to\infty$,  $\displaystyle\int_1^\infty \Big(M_{q_k}(t)+  M_{q_k}(t)^{1+\frac{2}{2q-n}}\Big)<\infty$, $\displaystyle\int_1^\infty M_{q_{k+1}}(t)^{\frac{n-2}{n}}\,\ud t <\infty$, and $M_{q_{k+1}}(t)\to 0$ as $t\to\infty$. Hence, using H\"older's inequality with \eqref{eq:BN-5},  $M_q(t)\to 0$ for any $q\ge \frac{n^2}{2(n-2)}$.

Finally, let us consider $n=3$. By \eqref{eq:diff-ineq-3}, we have 
\[
\frac{\ud }{\ud t} M_2(t) \le CM_2(t)(1+M_2(t)^{2}). 
\]
Using \eqref{eq:m-2}, we can pick $t_i\to \infty$ such that $M_2(t_i) \to 0$. Hence, 
\[
\arctan M_2(t) \le \arctan M_2(t_i) +C \int_{t_i}^\infty M_2(t)\,\ud t.  
\] It follows that $\lim_{t\to \infty}\arctan M_2(t) =0$ and thus $\lim_{t\to \infty} M_2(t)=0$. Hence, $\displaystyle\int_1^\infty M_{6}(t)^{\frac{1}{3}}\,\ud t <\infty$. Since $6>q_0$ when $n=3$, we can use the argument of those in Case 2 to show that $M_q(t)\to 0$ for all $q\ge 6$. By H\"older inequality, we conclude that $M_q(t)\to 0$ for all $q\ge 1$.
\end{proof}

\begin{cor} \label{cor:curvature-convergence}
We have
\[
\lim_{t\to\infty}\|\mathcal{R}-1\|_{L^\infty(\Omega)}=0.
\]
\end{cor}
\begin{proof}
Consider the equation of $1-\mathcal{R}$ as in the proof of Lemma \ref{lem:Rlowerbound}:
\[
\pa_t (1-\mathcal{R})= \frac{n-2}{n+2} \Delta_g (1-\mathcal{R}) +c(x,t) (1-\mathcal{R}) +\frac{4}{n+2} (1-\mathcal{R}),
\]
where $c(x,t)=-\frac{4}{n+2} (1-\mathcal{R})$. This is a linear equation of $1-\mathcal{R}$. We know from the proof of Proposition \ref{prop:Mq} that there exists a sufficiently large $q>1$ such that
\[
\int_1^\infty M_q(t)\,\ud t<\infty.
\]
This means that $c(x,t)$ has very high integrability against $\ud vol_g$ in space-time. Then we can apply the Moser's iteration as in the proof of Lemma \ref{lem:Rlowerbound} to obtain
\[
\|\mathcal{R}-1\|_{L^\infty(\Omega\times(T,\infty))} \le C \left(\int_{T-1}^\infty M_q(t)\,\ud t\right)^{\frac 1q}
\]
for all large $T$. Hence, the corollary follows. 
\end{proof}

\section{Concentration compactness}\label{sec:concentration}

The solution of \eqref{eq:BN-3} may blow up as $t\to \infty$ because of the critical exponent $\frac{n+2}{n-2}$. Nevertheless, we also know how the solutions may blow up. 

\begin{prop}\label{prop:c-c} Let $v$ be a solution of \eqref{eq:BN-3}. For any $t_\nu \to \infty$, $\nu \to \infty$, $v_\nu=v(\cdot,t_\nu)$ is  a Palais-Smale sequence of the functional $F$ given by \eqref{eq:F-funct} in $H_0^1(\om)$. 

\end{prop} 

\begin{proof}  We have already proved that $v_\nu$ is bounded in $H^1_0(\om)$ and $F(v_\nu)\to F_\infty$ as $\nu \to \infty$. It remains to show the derivative of $F$ at $v_\nu$ tends to zero. Indeed, for any $\varphi\in H^1_0(\om)$, we have 
{\allowdisplaybreaks
\begin{align*}
\langle \ud F(v_\nu ), \varphi \rangle&= 2\int_{\om} (-\Delta v_\nu-bv_\nu- v_\nu ^{\frac{n+2}{n-2}}) \varphi \,\ud x \\&
= 2\int_{\om} (\mathcal{R} -1 ) v_\nu ^{\frac{n+2}{n-2}} \varphi \,\ud x\\&
\le 2 \left(\int_{\om} |\mathcal{R} -1 |^{\frac{2n}{n+2}} v_\nu ^{\frac{2n}{n-2}} \,\ud x\right)^{\frac{n+2}{2n}} \left(\int_{\om} | \varphi|^{\frac{2n}{n-2}} \,\ud x \right)^{\frac{n-2}{2n}} \\&
\le C(n)M_{\frac{2n}{n+2}}(t_\nu)^{\frac{n+2}{2n} } \|\varphi \|_{H^1_0(\om)}, 
\end{align*}
}
where we used H\"older's inequality  and the Sobolev inequality. It follows from Proposition \ref{prop:Mq} that $\ud F(v_\nu)$ strongly converges to $0$ in $H^{-1}(\om)$. 

Therefore, the proof is completed. 
\end{proof}

The next proposition shows that the blow up points, if exist, will stay uniformly away from the boundary $\partial\Omega$.

\begin{prop} \label{prop:away} There exist two positive constants  $\delta_0$ and $C$, depending on $v(\cdot,1)$,   such that for all $x\in \om$ with $d(x):=\mbox{dist}(x,\partial\Omega)<\delta_0$ and $t\ge 1$,  
\[
v(x,t)\le C d(x). 
\]
\end{prop} 

\begin{proof} We are going to use the moving plane method as Han \cite{H} did for the elliptic case.  By the Hopf Lemma, there exist $\rho_0>0$ and $\al_0>0$ such that $v(z-\rho e, 1)$ is nondecreasing for $0<\rho<\rho_0$, where $z\in \pa \om$, $e\in \R^n $ with $|e|=1$, and $(e,\nu(z))\ge \al_0$ with $\nu(z)$ the unit out normal to $\pa \om $ at $z$. If $\om$ is strictly convex,  using the moving plane method we can conclude that  $v(z-\rho e, t)$ is nondecreasing for $0<\rho<\rho_0$, and for all $t\ge 1$. Therefore, we can find $\gamma>0 $ and $\delta>0$ such that for any fixed $t\ge 1$, and any $x\in\Omega$ satisfying $0<d(x)<\delta$, there exists a measurable set $\Gamma_x$  with (i) $meas(\Gamma_x )\ge \gamma$, (ii) $\Gamma_x \subset \{z: d(z)\ge \delta/2 \}$, and (iii) $v(y,t)\ge v(x,t)$ for any $y\in \Gamma_x$. Actually, $\Gamma_x$ can be taken to a piece of cone with vertex at $x$. It follows that for any $x\in \{z:0<d(z)<\delta\}$, we have 
\[
v(x,t) \le \frac{1}{meas(\Gamma_x )} \int_{\Gamma_x}   v(y, t)\,\ud y \le \frac{C}{\gamma},
\]
where we used \eqref{eq:BN-5} and H\"older's inequality. Namely, $v(x,t)\le C$ for $(x,t)\in \{z:0<d(z)<\delta\} \times [1,\infty)$.  By the proof of Theorem 4.1 in \cite{DKV}, we have $v(x,t)\le Cd(x)$ for $(x,t)\in \{z:0<d(z)<\delta\} \times [1,\infty)$. 

For a general domain, one can first use a Kelvin transform near each boundary point, and then apply the moving plane method. Pick any point $P\in \pa \om$ for instance. Since we assume the boundary of the domain $\om$ is smooth, we may assume, without loss of generality, that the unit ball $B_1$ contacts $P$ from the exterior of $\om$ (i.e., $B_1\subset\Omega^c$ and $P\in\partial B_1$). Let $w(x,t)$ be the Kelvin transform of $v$: 
\[
w(x,t)= |x|^{2-n} v\left(\frac{x}{|x|^2}, t\right). 
\] 
Then 
\[
\begin{cases}
\pa_t w^{\frac{n+2}{n-2}} =\Delta w+b|x|^{-4} w +w^{\frac{n+2}{n-2}} \quad \mbox{in }\om_P \times (0,\infty)\\  
w=0 \quad \mbox{on }\pa \om_P \times (0,\infty),
\end{cases}
\]
where $\om_P$ is the image of $\om$ under the Kelvin transform. Since $b\ge 0$,  $b|x|^{-4}$ is nondecreasing along the $-P$ direction. Applying the moving plane method we have that  $w(\cdot,t)$ is nondecreasing along the $-P$ direction in a neighborhood (uniform in $t$) of $P$. Since the $L^{\frac{2n}{n-2}}$ norm is invariant under the Kelvin transform, using the above argument in the case of strictly convex domains, we conclude that $w(\cdot, t)$ is bounded in a neighborhood of $P$ independent of $t$ and so is $v(\cdot, t)$.  It follows that $v(x,t)\le Cd(x)$ for $(x,t)\in \{z:0<d(z)<\delta\} \times [1,\infty)$ for some $\delta>0$.

Therefore, the proof is completed. 
\end{proof}

 For $a\in \R^n$ and $\lda\in (0,\infty)$, let 
\be \label{eq:bubble-def}
\bar \xi_{a,\lda}(x)=c_0 \left(\frac{\lda}{1+\lda^2|x-a|^2}\right)^{\frac{n-2}{2}}
\ee
with $c_0=(n(n-2))^{\frac{n-2}{4}}$. Then we have 
\[
-\Delta \bar \xi_{a,\lda} = \bar \xi_{a,\lda}^{\frac{n+2}{n-2}} \quad \mbox{in }\R^n
\]
and 
\[
\int_{\R^n} \bar \xi_{a,\lda} ^{\frac{2n}{n-2}}= Y(\Sn)^{\frac{n}{2}},
\]
where $\Sn$ is the standard unit sphere in $\R^{n+1}$, 
\[
Y(\Sn)= \frac{n(n-2)}{4} |\Sn|^{\frac{2}{n}}=\inf_{u\in H^{1}(\mathbb{S}^n )} \frac{\int_{\Sn } |\nabla u|^2 +\frac{n(n-2)}{4} u^2\ud vol_{g_{\mathbb{S}^n}}}{(\int_{\Sn} | u|^{\frac{2n}{n-2}} vol_{g_{\mathbb{S}^n}} )^{\frac{n-2}{n}}}
\] 
and $|\Sn|$ is the area of $\Sn$. Define
 \be \label{eq:xi}
\xi_{a,\lda}(x)=  \bar \xi_{a,\lda}(x)-h_{a,\lda}(x), 
\ee
where $\Delta h_{a,\lda}(x) = 0$ in $\om$ and $h_{a,\lda}= \bar \xi_{a,\lda} $ on $\pa \om$.  By the maximum principle, $\xi_{a,\lda}>0$ in $\Omega$ and $h_{a,\lda}>0$ in $\overline\Omega$. 
 
\begin{prop}\label{prop:s-b-c} Let $v$ be a solution of \eqref{eq:BN-3}.  For any $t_\nu \to \infty$, $\nu \to \infty$,  after passing to a subsequence if necessary, $v_\nu$ weakly converges to $v_\infty$ in $H_0^1(\om)$ and  we can find a non-negative integer $m$ and a sequence of m-tuplets $(x^*_{k,\nu}, \lda_{k,\nu}^*)_{1\le k\le m}$, $(x^*_{k,\nu}, \lda_{k,\nu}^*) \in \om \times (0,\infty)$, with the following properities.  
\begin{enumerate}
 
 \item The function $v_\infty\in H_0^1(\om)$ satisfies the equation $-\Delta v_\infty-bv_\infty= v_\infty^{\frac{n+2}{n-2}} $ in $\om$.

\item There hold, for all $i\neq j$, 
\[
\frac{ \lda_{i, \nu}^* }{\lda_{j, \nu}^* }+\frac{ \lda_{j, \nu}^* }{\lda_{i, \nu}^* }+\lda_{i, \nu}^* \lda_{j, \nu}^* |x^*_{i,\nu}-x^*_{j,\nu} |^2 \to \infty,
\]
and for all $k$, $d(x^*_{k,\nu})\ge \delta_0/2$ with the constant $\delta_0>0$ in  Proposition \ref{prop:away}, 
\[
\lda_{k, \nu}^* d(x^*_{k,\nu}) \to \infty 
\]
as $\nu \to \infty$. 
 
\item We have 
\[
\left\|v_\nu- v_\infty-\sum_{k=1}^m \xi_{x^*_{k,\nu}, \lda_{k,\nu}^*}\right\| \to 0
\]
as $\nu\to \infty$. 

\item We have 
\[
F(v_\nu) = F(v_\infty) +\frac{2m}{n} Y(\mathbb{S}^n)^{n/2} +o(1),
\]
where $o(1)\to 0$ as $\nu\to \infty$.

\end{enumerate}

\end{prop}

\begin{proof} 
This proposition follows from Propositions \ref{prop:c-c}, and the  compactness result of Brezis-Coron \cite{BC85} and Struwe \cite{St}. More precisely, the proposition except item 2 follows from Proposition 2.1 in Struwe \cite{St}.  By Proposition \ref{prop:away}, $d(x^*_{k,\nu})\ge \delta_0/2$ with the same $\delta_0/2>0$. Namely, the energy can not concentrate at a fixed neighborhood of the boundary.  By  Theorem 2 in Brezis-Coron \cite{BC85} or Proposition 4 in  Bahri-Coron \cite{Bahri-C}, we have, for all $i\neq j$, 
\[
\frac{ \lda_{i, \nu}^* }{\lda_{j, \nu}^* }+\frac{ \lda_{j, \nu}^* }{\lda_{i, \nu}^* }+\lda_{i, \nu}^* \lda_{j, \nu}^* |x^*_{i,\nu}-x^*_{j,\nu} |^2 \to \infty,
\]
and for all $k$ and $\lda_{k, \nu}^*  \to \infty $ as $\nu \to \infty$.  This is item 2.
\end{proof} 

A similar result for the harmonic map heat flow was proved by  Qing-Tian \cite{QT}.
The correction term $h_{a,\lambda}$ in \eqref{eq:xi} is small and can be controlled.

\begin{lem} \label{lem:bubble-derivative}  Let $\xi_{a,\lda}$ and $h_{a,\lda}$ be defined as in \eqref{eq:xi}. Suppose $a\in \om$ with $d(a)>\delta>0$ and $\lda>1$. Then we have, for  $ x\in \om,$
\[
|h_{a,\lda}(x)|+|\nabla_{a} h_{a,\lda}(x)|+\lda |\nabla_{\lda} h_{a,\lda}(x)| \le C(n, \om,\delta) \lda^{-\frac{n-2}{2}}, 
\]
\[
\nabla_{a} \xi_{a,\lda} (x)=(n-2)\xi_{a,\lda} \frac{\lda^2(x-a)}{1+\lda^2|x-a|^2} +O(\lda^{-\frac{n-2}{2}}),
\]
and 
\[
\nabla_{\lda} \xi_{a,\lda} (x)=\frac{(n-2)}{2\lda } \xi_{a,\lda} \frac{1-\lda^2|x-a|^2}{1+\lda^2 |x-a|^2} +O(\lda^{-\frac{n}{2}}),
\]
where $|O(\lda^{-\frac{n-2}{2}}) |\le C \lda^{-\frac{n-2}{2}}$ for some $C$ depending only on $n, \om$ and $\delta$.
\end{lem}

\begin{proof} 
Since $\Delta h_{a,\lda}(x) = 0$ in $\om$ and $h_{a,\lda}= \bar \xi_{a,\lda} $ on $\pa \om$, the estimate of $h_{a,\lda}$  follows from the Poisson formula for the Laplace equation. Then,
\begin{align*}
\nabla_{a} \xi_{a,\lda} (x)&=\nabla_{a} \bar \xi_{a,\lda} (x) -\nabla_{a} h_{a,\lda}\\&
=  (n-2)\bar \xi_{a,\lda} \frac{\lda^2(x-a)}{1+\lda^2|x-a|^2}+O(\lda^{-\frac{n-2}{2}})\\&
= (n-2) \xi_{a,\lda} \frac{\lda^2(x-a)}{1+\lda^2|x-a|^2}+O(\lda^{-\frac{n-2}{2}}).
\end{align*}
The estimate $\nabla_{\lda} \xi_{a,\lda} (x)$ can be obtained similarly.
\end{proof}

\section{Refined blow up analysis}\label{sec:blowup}

We continue from Proposition \ref{prop:s-b-c}. By the strong maximum principle, the nonnegative limit $v_\infty$ either is positive in $\om$ or identically equals to zero. We will treat these two cases separably in two subsections.  We will adapt the refined blow up analysis in Brendle \cite{Br05} to the equation \eqref{eq:BN-3}.  

\subsection{The case $v_\infty \equiv 0$}\label{sec: zero}

First, we shall project $v_\nu$ to an $m(n+2)$-dimensional surface in $H_0^1(\om)$ generated by $m$-bubbles. For every $\nu$, let $\mathcal{A}_\nu$ be the closed set of all $m$-tuplets $(x_k,\lda_k, \al_k)_{1\le k\le m}$ satisfying $(x_k,\lda_k, \al_k) \in \overline B_{\frac{1}{\lda_{k,\nu}^*}}(x_{k,\nu}^*) \times [\frac{\lda_{k,\nu}^*}{2}, \frac{3\lda_{k,\nu}^*}{2} ] \times [\frac{1}{2},\frac32]$.  Choose an $m$-tuplet $(x_{k,\nu},\lda_{k,\nu}, \al_{k,\nu})_{1\le k\le m}\in \mathcal{A}_\nu$ such that  
\be\label{eq:finite-minimum}
\left\|v_\nu- \sum_{k=1}^m \al_{k,\nu} \xi_{x_{k,\nu}, \lda_{k,\nu}}\right\| =\inf_{(x_k,\lda_k, \al_k)_{1\le k\le m} \in \mathcal{A}_{\nu}} \left\|v_\nu- \sum_{k=1}^m \al_{k} \xi_{x_{k}, \lda_{k}}\right\|.
\ee
By Proposition \ref{prop:s-b-c}, Proposition \ref{prop:away} and the definition of $(x_{k,\nu},\lda_{k,\nu}, \al_{k,\nu})_{1\le k\le m}$, we have, for all $i\neq j$, 
\be\label{eq:bubbles-1}
\frac{ \lda_{i, \nu} }{\lda_{j, \nu}}+\frac{ \lda_{j, \nu} }{\lda_{i, \nu}}+\lda_{i, \nu} \lda_{j, \nu} |x_{i,\nu}-x_{j,\nu} |^2 \to \infty,
\ee
and for all $k$
\be\label{eq:bubbles-2}
\lda_{k, \nu} d(x_{k,\nu}) \to \infty 
\ee
as $\nu \to \infty$. 
In addition, $d(x_{k,\nu})>\delta_0/2$ with same $\delta_0$ in Proposition \ref{prop:away}, and 
\be\label{eq:bubbles-3}
\left\|v_\nu-\sum_{k=1}^m  \al_k\xi_{x_{k,\nu}, \lda_{k,\nu}}\right\| \to 0
\ee
as $\nu\to \infty$. 

By the triangle inequality, 
\begin{align*}
&\left\|\sum_{k=1}^m  \al_k\xi_{x_{k,\nu}, \lda_{k,\nu}}- \sum_{k=1}^m \xi_{x^*_{k,\nu}, \lda_{k,\nu}^*} \right\| \\& \le \left\|v_\nu-\sum_{k=1}^m  \al_k\xi_{x_{k,\nu}, \lda_{k,\nu}} \right\|+\left\|v_\nu- \sum_{k=1}^m \xi_{x^*_{k,\nu}, \lda_{k,\nu}^*} \right\|=o(1).
\end{align*}
It follows that, for all $1\le k\le m$,  
\be\label{eq:bubbles-4}
|x_{k,\nu}- x_{k,\nu}^{*}|=o(1) \frac{1}{\lda_{k,\nu}^*},\quad  \frac{\lda_{k,\nu}}{\lda_{k,\nu}^*}=1+o(1), \quad \al_{k,\nu}=1+o(1).  
\ee
In particular, $(x_{k,\nu},\lda_{k,\nu}, \al_{k,\nu})_{1\le k\le m}$ is an interior point of $\mathcal{A}_{\nu}$.  

In the sequel, we assume 
\be \label{eq:order}
\lda_{1,\nu}\ge \lda_{2,\nu} \ge \dots \ge \lda_{m,\nu}.
\ee
 Let 
\be \label{eq:decomposition-1}
U_\nu= \sum_{k=1}^m  \al_{k,\nu}\xi_{x_{k,\nu}, \lda_{k,\nu}}, \quad w_\nu=v_{\nu}-U_{\nu}. 
\ee
Next, we shall estimate the orthogonal part  $w_\nu$ of the above projection. 

\begin{lem} \label{lem:proj-1} We have for $1\le k\le m$,  
\begin{align*}
&\Big|\int_\om \xi_{x_{k,\nu}, \lda_{k,\nu}}^{\frac{n+2}{n-2}}  w_\nu \,\ud x\Big|+\Big|\int_\om \xi_{x_{k,\nu}, \lda_{k,\nu}}^{\frac{n+2}{n-2}}  \frac{1-\lda^2|x-x_{k,\nu}|^2}{1+\lda^2|x-x_{k,\nu}|^2} w_\nu \,\ud x\Big |\\[2mm] &\quad  +\Big|\int_\om \xi_{x_{k,\nu}, \lda_{k,\nu}}^{\frac{n+2}{n-2}}  \frac{\lda^2(x-x_{k,\nu})}{1+\lda^2|x-x_{k,\nu}|^2} w_\nu \,\ud x\Big | \le o(1) \Big(\int_{\om} |w_\nu|^{\frac{2n}{n-2}}\,\ud x\Big)^{\frac{n-2}{2n}}
\end{align*}
\end{lem} 

\begin{proof} By  the finite dimensional variational problem  \eqref{eq:finite-minimum} and \eqref{eq:bubbles-4}, taking derivatives in $\mathcal{A}_\nu$, we have 
\[
\int_{\om}\Big[\nabla (\nabla_{a,\lda}\xi_{x_{k,\nu},\lda_{k,\nu}}) \nabla w_\nu-b\nabla_{a,\lda}\xi_{x_{k,\nu},\lda_{k,\nu}} w_\nu\Big]\,\ud x=0
\]
and 
\[
\int_{\om}\Big[\nabla \xi_{x_{k,\nu},\lda_{k,\nu}} \nabla w_\nu-b\xi_{x_{k,\nu},\lda_{k,\nu}} w_\nu\Big]\,\ud x=0,
\]
where $\nabla_{a,\lda}\xi_{x_{k,\nu},\lda_{k,\nu}}= \nabla_{a,\lda}\xi_{a,\lda}\Big|_{(a,\lda)=(x_{k,\nu},\lda_{k,\nu})}$.  Integrating by parts, using the equation of $\bar \xi_{a,\lda}$, H\"older's inequality and Lemma \ref{lem:bubble-derivative}, the lemma follows. 
\end{proof} 

Note that the bubbles are non-degenerate, since we have the following well known lemma (see (3.14) in Rey \cite{Rey}). 

\begin{lem}\label{lem:coercive} Let $\bar \xi_{a,\lda}$ be defined in \eqref{eq:bubble-def}. Then there exists a constant $c_1>0$ depending only on $n$ such that 
\[
(1-c_1)\int_{\R^n} |\nabla \varphi|^2 \ge \frac{n+2}{n-2}\int_{\R^n} \bar \xi_{0,1}^\frac{4}{n-2} \varphi^2
\]
for any $\varphi\in H^{1}_0(\R^n)$ satisfying
\[
\int_{\R^n}  \bar \xi_{0,1}^\frac{4}{n-2} (\nabla_{a,\lda} \bar \xi_{0,1}) \varphi\,\ud x=0.
\]
\end{lem}
We have the following non-degeneracy estimates of the second variation of $F$ for $w_\nu$.
\begin{lem} \label{lem:2nd-variational} For large $\nu$, we have 
\[
\frac{n+2}{n-2} \int_\om \sum_{k=1}^m  \xi_{x_{k,\nu},\lda_{k,\nu}}  ^{\frac{4}{n-2}} w_\nu^2   \le (1-c) \int_{\om} (|\nabla w_\nu|^2-bw_\nu^2)\,\ud x,
\]
where $c>0$ is independent of $\nu$.

\end{lem}

\begin{proof} We assume $w_\nu$ is not zero, otherwise there is nothing to prove. Define $\tilde w_\nu= \frac{w_\nu}{\|w_\nu\|}$. Suppose the lemma is not true.  Then we can find a subsequence of $\{\tilde w_\nu\}$ (still denoted by $\{\tilde w_\nu\}$) satisfying
\be \label{eq:coer-1}
\lim_{\nu\to \infty}\frac{n+2}{n-2}  \int_\om   \sum_{k=1}^m \xi_{x_{k,\nu},\lda_{k,\nu}}  ^{\frac{4}{n-2}} \tilde w_\nu^2  \ge 1. 
\ee
By \eqref{eq:sobolev-b},  
\be \label{eq:coer-2}
\int_{\om} |\tilde w_\nu|^{\frac{2n}{n-2}} \le K_b^{\frac{n}{n-2}} \|\tilde w_\nu\|= K_b^{\frac{n}{n-2}}. 
\ee
By \eqref{eq:bubbles-1} and \eqref{eq:order}, we can find $R_\nu\to \infty$, $R_\nu \lda_{j,\nu}^{-1} \to 0 $ for all $1\le j\le m$, and 
\be \label{eq:coer-3}
\frac{\lda_{i,\nu}}{R_\nu} (\lda_{j,\nu}^{-1} +|x_{i,\nu}-x_{j,\nu}|) \to \infty 
\ee
for all $i<j$. Set 
\[
\om_{j,\nu}= B_{R_\nu \lda_{j,\nu}^{-1}}(x_{j,\nu})\setminus \bigcup_{i=1}^{j-1} B_{R_\nu \lda_{i,\nu}^{-1}}(x_{i,\nu}).
\]
By \eqref{eq:coer-1} and $\|\tilde w_\nu\|=1$, we can find $1\le j\le m$ such that 
\[
\lim_{\nu\to \infty } \int_\om    \xi_{x_{j,\nu},\lda_{j,\nu}}  ^{\frac{4}{n-2}} \tilde w_\nu^2 >0
\]
and 
\[
\lim_{\nu\to \infty  } \int_{\om_{j,\nu}} (|\nabla \tilde w_\nu|^2- b\tilde w_\nu^2 ) \le \lim_{\nu\to \infty  }\frac{n+2}{n-2}  \int_\om    \xi_{x_{j,\nu},\lda_{j,\nu}}  ^{\frac{4}{n-2}} \tilde w_\nu^2. 
\]
Let $\hat w_\nu(x)= \lda_{j,\nu}^{-\frac{n-2}{2}} \tilde w_\nu(x_{j,\nu} +\lda_{j,\nu}^{-1} x)$. Under this scaling, by using \eqref{eq:coer-3}, we know that either $B_{R_\nu \lda_{j,\nu}^{-1}}(x_{j,\nu})$ will be disjoint with $\bigcup_{i=1}^{j-1} B_{R_\nu \lda_{i,\nu}^{-1}}(x_{i,\nu})$, or $B_{R_\nu \lda_{i,\nu}^{-1}}(x_{i,\nu})$ will shrink to a point for every $1\le i\le j-1$.  By passing to a weak limit in $H^1_{loc}(\R^n)$, and using the above two inequalities and Lemma \ref{lem:proj-1}, we then obtain a contradiction to Lemma \ref{lem:coercive}. 

Therefore, Lemma \ref{lem:2nd-variational} is proved. 
\end{proof} 

\begin{cor}\label{cor:coercive} 
For large $\nu$, we have 
\[
\frac{n+2}{n-2} \int_\om U_\nu ^{\frac{4}{n-2}} w_\nu^2   \le (1-c) \int_{\om} (|\nabla w_\nu|^2-bw_\nu^2)\,\ud x,
\]
where $c>0$ is independent of $\nu$.
\end{cor}
\begin{proof}
It follows from Lemma \ref{lem:2nd-variational}, H\"older's inequality, the Sobolev inequality \eqref{eq:sobolev-b} and the fact that
\[
\int_\Omega \Big|U_\nu ^{\frac{4}{n-2}} - \sum_{k=1}^m  \xi_{x_{k,\nu},\lda_{k,\nu}}  ^{\frac{4}{n-2}}\Big|^{\frac{n}{2}}=o(1).
\]
\end{proof}

Now we can have an expansion of the Hamitonian $F$ defined in \eqref{eq:F-funct}.
\begin{prop} \label{prop:bubble-energy} When $n\ge 4$ and  $\nu$ is sufficiently large,  we have 
\[
F(U_\nu) \le  \sum_{k=1}^m F( \xi_{x_{k,\nu},\lda_{k,\nu}}  ) +o(1) \sum_{k=1}^m  \int_{\om} \xi_{x_{k,\nu},\lda_{k,\nu}}^2   +C \sum_{k=1}^m \lda_{k,\nu}^{2-n}. 
\]
\end{prop} 

\begin{proof}  We shall need the following inequality
\be\label{eq:basic}
\left(\sum_{k=1}^m a_k \right)^{\frac{2n}{n-2}} \ge \sum_{k=1}^m a_k ^{\frac{2n}{n-2}}+\frac{2n}{n-2}\sum_{k< l} a_k^{\frac{n+2}{n-2}} a_l + c_{n,m}\sum_{k< l} (a_k \lor a_l) ^{\frac{4}{n-2}} (a_k\land a_l)^{2}
\ee for  any $a_1,\dots, a_m\ge 0$, where  $c_{n,m}>0$ is a constant, and $a_k \lor a_l=\max(a_k,a_l)$ and $a_k\land a_l=\min(a_k,a_l)$. This inequality can be proved using Lemma \ref{lem:basic-1} and induction.

Using the  inequality \eqref{eq:basic},  we have 
\begin{align}
&\int_\Omega (| \nabla U_\nu|^2 -bU_\nu^2) -\frac{n-2}{n} \int_\Omega U_\nu^{\frac{2n}{n-2}} \nonumber\\
&\le \sum_{k} \al_{k,\nu}^2 \int_\Omega ( |\nabla \xi_{x_{k,\nu},\lda_{k,\nu}} |^2-b\xi_{x_{k,\nu},\lda_{k,\nu}} ^2)- \sum_{k} \al_{k,\nu}^{\frac{2n}{n-2}} \frac{n-2}{n}\int_\Omega\xi_{x_{k,\nu},\lda_{k,\nu}}^{\frac{2n}{n-2}} \nonumber\\
&\quad +2 \sum_{i<j}  \al_{j,\nu} \left[\al_{i,\nu} \int_\Omega (\nabla \xi_{x_{i,\nu},\lda_{i,\nu}} \nabla \xi_{x_{j,\nu},\lda_{j,\nu}} -b\xi_{x_{i,\nu},\lda_{i,\nu}} \xi_{x_{j,\nu},\lda_{j,\nu}} )-  \al_{i,\nu}^{\frac{n+2}{n-2}}  \int_\Omega   \xi_{x_{i,\nu},\lda_{i,\nu}} ^{\frac{n+2}{n-2}}\xi_{x_{j,\nu},\lda_{j,\nu}}\right] \nonumber\\
&\quad-c_{n,m} \sum_{i<j}  \int_\Omega   (\xi_{x_{i,\nu},\lda_{i,\nu}} \lor \xi_{x_{j,\nu},\lda_{j,\nu}})^{\frac{4}{n-2}} (\xi_{x_{i,\nu},\lda_{i,\nu}} \land \xi_{x_{j,\nu},\lda_{j,\nu}})^{2}.\label{eq:computeFU1} 
\end{align}
By the equation of $\bar \xi_{x_{k,\nu},\lda_{k,\nu}}$ and the definition of $\xi_{x_{k,\nu},\lda_{k,\nu}}$, we have 
\begin{align}
&\al_{k,\nu}^2 \int_\Omega ( |\nabla \xi_{x_{k,\nu},\lda_{k,\nu}} |^2-b\xi_{x_{k,\nu},\lda_{k,\nu}} ^2)- \al_{k,\nu}^{\frac{2n}{n-2}} \frac{n-2}{n}\int_\Omega\xi_{x_{k,\nu},\lda_{k,\nu}}^{\frac{2n}{n-2}}  \nonumber\\
&\le \al_{k,\nu}^2 \int_{\R^n}  |\nabla \bar \xi_{x_{k,\nu},\lda_{k,\nu}} |^2- \al_{k,\nu}^{\frac{2n}{n-2}} \frac{n-2}{n}\int_{\R^n} \bar \xi_{x_{k,\nu},\lda_{k,\nu}}^{\frac{2n}{n-2}}  - b \al_{k,\nu}^2 \int_\Omega  \xi_{x_{k,\nu},\lda_{k,\nu}} ^2  +C \lda_{k,\nu}^{2-n}\nonumber\\
&\le \int_{\R^n}  |\nabla \bar \xi_{x_{k,\nu},\lda_{k,\nu}} |^2 - \frac{n-2}{n}\int_{\R^n}  \bar \xi_{x_{k,\nu},\lda_{k,\nu}}^{\frac{2n}{n-2}}  -\frac{4}{n-2}(\al_{k,\nu}-1)^2 Y(\Sn)^{\frac{n}{2}} \nonumber\\
&\quad -b \al_{k,\nu}^2 \int_\Omega  \xi_{x_{k,\nu},\lda_{k,\nu}} ^2 +C \lda_{k,\nu}^{2-n}\nonumber\\
&\le  F( \xi_{x_{k,\nu} ,\lda_{k,\nu}}) -\frac{4}{n-2}(\al_{k,\nu}-1)^2Y(\Sn)^{\frac{n}{2}}+o(1) \int_\Omega  \xi_{x_{k,\nu},\lda_{k,\nu}} ^2 +C\lda_{k,\nu}^{2-n},\label{eq:computeFU2} 
\end{align}
where we used $\al_{k,\nu}=1+o(1)$, $\al_{k,\nu}^2-\frac{n-2}{n} \al_{k,\nu}^{\frac{2n}{n-2}}\le  \frac{2}{n}-\frac{4}{n-2}(\al_{k,\nu}-1)^2 $ and 
\[
\int_{\R^n} |\nabla \bar \xi_{x_{k,\nu},\lda_{k,\nu}} |^2= \int_{\R^n} \bar  \xi_{x_{k,\nu},\lda_{k,\nu}} ^{\frac{2n}{n-2}}=Y(\Sn)^{\frac{n}{2}}.
\]
In addition, 
\begin{align}
& \al_{i,\nu}  \int_\Omega (\nabla \xi_{x_{i,\nu},\lda_{i,\nu}} \nabla \xi_{x_{j,\nu},\lda_{j,\nu}} -b\xi_{x_{i,\nu},\lda_{i,\nu}} \xi_{x_{j,\nu},\lda_{j,\nu}} ) -\al_{i,\nu}^{\frac{n+2}{n-2}}  \int_\Omega   \xi_{x_{i,\nu},\lda_{i,\nu}} ^{\frac{n+2}{n-2}}\xi_{x_{j,\nu},\lda_{j ,\nu}} \nonumber\\
& = \al_{i,\nu} \int_\Omega (-\Delta \xi_{x_{i,\nu},\lda_{i,\nu}}  -b\xi_{x_{i,\nu},\lda_{i,\nu}} -\al_{i,\nu}^{\frac{4}{n-2}}  \xi_{x_{i,\nu},\lda_{i,\nu}} ^{\frac{n+2}{n-2}} )\xi_{x_{j,\nu},\lda_{j,\nu}} \nonumber\\&
\le C|  \al_{i,\nu} -1| \int_\Omega  \xi_{x_{i,\nu},\lda_{i,\nu}} ^{\frac{n+2}{n-2}} \xi_{x_{j,\nu},\lda_{j,\nu}}  + \int_\Omega  \Big(\bar\xi_{x_{i,\nu},\lda_{i,\nu}} ^{\frac{n+2}{n-2}}- \xi_{x_{i,\nu},\lda_{i,\nu}} ^{\frac{n+2}{n-2}}\Big) \xi_{x_{j,\nu},\lda_{j,\nu}} \nonumber\\&
\le C|  \al_{i,\nu} -1| \int_\Omega  \xi_{x_{i,\nu},\lda_{i,\nu}} ^{\frac{n+2}{n-2}} \xi_{x_{j,\nu},\lda_{j,\nu}}  +O(\lda_{i,\nu} ^{2-n}+\lda_{j,\nu}^{2-n}) \nonumber\\
& \le \frac{2}{n-2}(\al_{k,\nu}-1)^2  Y(\Sn)^{\frac{n}{2}}+C\Big (\int_\Omega  \xi_{x_{i,\nu},\lda_{i,\nu}} ^{\frac{n+2}{n-2}} \xi_{x_{j,\nu},\lda_{j,\nu}}\Big)^2 +O(\lda_{i,\nu} ^{2-n}+\lda_{j,\nu}^{2-n}),\label{eq:computeFU3} 
\end{align}
where $C>0$ is independent of $\nu$, and in the second inequality we used
\begin{align*}
\int_\Omega  \Big(\bar\xi_{x_{i,\nu},\lda_{i,\nu}} ^{\frac{n+2}{n-2}}- \xi_{x_{i,\nu},\lda_{i,\nu}} ^{\frac{n+2}{n-2}}\Big) \xi_{x_{j,\nu},\lda_{j,\nu}} 
&\le C \int_\Omega  \bar\xi_{x_{i,\nu},\lda_{i,\nu}} ^{\frac{4}{n-2}}|h_{x_{i,\nu},\lda_{i,\nu}} |  \bar\xi_{x_{j,\nu},\lda_{j,\nu}} \\
&\le C\lda_{i,\nu}^{\frac{2-n}{2}} \left(\int_\Omega  \bar\xi_{x_{i,\nu},\lda_{i,\nu}} ^{\frac{n+2}{n-2}}\right)^{\frac{4}{n+2}} \left(\int_\Omega\bar\xi_{x_{j,\nu},\lda_{j,\nu}} ^{\frac{n+2}{n-2}}\right)^{\frac{n-2}{n+2}}\\
&\le C\lda_{i,\nu}^{\frac{2-n}{2}}\lda_{i,\nu}^{\frac{2-n}{2}\frac{4}{n+2}}\lda_{j,\nu}^{\frac{2-n}{2}\frac{n-2}{n+2}}\\
&\le C\lda_{i,\nu}^{\frac{2-n}{2}}(\lda_{i,\nu}^{\frac{2-n}{2}}+\lda_{j,\nu}^{\frac{2-n}{2}})\\
&\le C(\lda_{i,\nu}^{2-n}+\lda_{j,\nu}^{2-n}).
\end{align*}
Combining \eqref{eq:computeFU1}, \eqref{eq:computeFU2} and  \eqref{eq:computeFU3}, we have
\begin{align*}
&F(U_\nu) \\
&\le  \sum_{k=1}^m F( \xi_{x_{k,\nu},\lda_{k,\nu}}  ) +o(1) \sum_{k=1}^m  \int_{\om} \xi_{x_{k,\nu},\lda_{k,\nu}}^2   +C \sum_{k=1}^m \lda_{k,\nu}^{2-n}\\
& \ \ + \sum_{i<j} \left[C\Big (\int_\Omega  \xi_{x_{i,\nu},\lda_{i,\nu}} ^{\frac{n+2}{n-2}} \xi_{x_{j,\nu},\lda_{j,\nu}}\Big)^2 - c_{n,m}  \int_\Omega   (\xi_{x_{i,\nu},\lda_{i,\nu}} \lor \xi_{x_{j,\nu},\lda_{j,\nu}})^{\frac{4}{n-2}} (\xi_{x_{i,\nu},\lda_{i,\nu}} \land \xi_{x_{j,\nu},\lda_{j,\nu}})^{2}\right].
\end{align*}
Meanwhile, we have
\[
\begin{split}
&C\Big (\int_\Omega  \xi_{x_{i,\nu},\lda_{i,\nu}} ^{\frac{n+2}{n-2}} \xi_{x_{j,\nu},\lda_{j,\nu}}\Big)^2 - c_{n,m}  \int_\Omega   (\xi_{x_{i,\nu},\lda_{i,\nu}} \lor \xi_{x_{j,\nu},\lda_{j,\nu}})^{\frac{4}{n-2}} (\xi_{x_{i,\nu},\lda_{i,\nu}} \land \xi_{x_{j,\nu},\lda_{j,\nu}})^{2} \\
& \le C\Big (\int_{\R^n} \bar \xi_{x_{i,\nu},\lda_{i,\nu}} ^{\frac{n+2}{n-2}} \bar\xi_{x_{j,\nu},\lda_{j,\nu}}\Big)^2 - c_{n,m}  \int_{\R^n}   (\bar\xi_{x_{i,\nu},\lda_{i,\nu}} \lor \bar\xi_{x_{j,\nu},\lda_{j,\nu}})^{\frac{4}{n-2}} (\bar\xi_{x_{i,\nu},\lda_{i,\nu}} \land \bar\xi_{x_{j,\nu},\lda_{j,\nu}})^{2} \\
&\quad+C(\lda_{i,\nu}^{2-n}+\lda_{j,\nu}^{2-n})\\
&\le C(\lda_{i,\nu}^{2-n}+\lda_{j,\nu}^{2-n})\quad\mbox{for all large }\nu,
\end{split}
\]
where we used \eqref{eq:I1I2} in the last inequality. 

Therefore, the proof is completed. 
\end{proof}

\begin{cor}\label{cor:pure-singular} 
If $n\ge 4$ and $b>0$ satisfying \eqref{eq:b}, we have, for large $\nu$,  
\[
F(U_\nu) \le \frac{2m}{n} Y(\Sn)^{\frac{n}{2}}. 
\]

\end{cor} 

\begin{proof} 
By Proposition \ref{prop:away} and Lemma \ref{lem:bubble-derivative},  \begin{align*}
F( \xi_{x_{k,\nu},\lda_{k,\nu}}  ) &= \int_{\om } \Big( |\nabla \xi_{x_{k,\nu},\lda_{k,\nu}} |^2  -\frac{n-2}{n} \xi_{x_{k,\nu},\lda_{k,\nu}}^{\frac{2n}{n-2}}\Big)\,\ud x- b\int_{\om }  \xi_{x_{k,\nu},\lda_{k,\nu}}^2\,\ud x\\&\le \frac{2}{n} Y(\Sn)^{\frac{n}{2}}  + C \lda_{k,\nu}^{2-n} - b\int_{\om }  \xi_{x_{k,\nu},\lda_{k,\nu}}^2\,\ud x,
\end{align*}
where $C>0$ is independent of $\nu$.
Note that 
\[
\int_{\om}\xi_{x_{j,\nu},\lda_{j,\nu}}^{2}\ge\begin{cases}  \frac{1}{C} \lda_{j,\nu}^{-\frac{4}{n-2}},& \quad \mbox{if }n\neq 4, \\ 
 \frac{1}{C}  \lda_{j,\nu}^{-2} \ln \lda_{j,\nu},& \quad \mbox{if }n=4.
\end{cases}
\]
Hence, if $n\ge 4$ and $b>0$, for any large constant $N$ we can find $j_N>0$ such that for all $j\ge j_N$ there holds   $b\int_{\om}\xi_{x_{j,\nu},\lda_{j,\nu}}^{2} \ge N \lda_{j,\nu}^{2-n}$.
The corollary follows immediately from Proposition \ref{prop:bubble-energy}. 
\end{proof}

\subsection{The case $v_\infty >0$}

In this case, we shall also project $v_\nu$ to a finite-dimensional surface in $H_0^1(\om)$ generated by $v_\infty$ and $m$-bubbles. In order to understand the new contribution from $v_\infty$, we need to perform spectral analysis  of the linearized operator at $v_\infty$ as Brendle \cite{Br05} did for the Yamabe flow on compact manifolds.  Our current $H_0^1(\om)$  setting is more close to that in Section 2.1 of Bonforte-Figalli \cite{BFig}. Indeed, the analysis of \cite{BFig} applies here with little change and the election of $L$ below is the same as $k_p$ in \cite{BFig}.

Let $\mathcal{L}^2(\Omega):=\big\{f: \int_{\Omega}f^2v_\infty^{\frac{4}{n-2}}<\infty\big\}$ with the inner product $\langle f,g\rangle=\int_{\Omega}fgv_\infty^{\frac{4}{n-2}}\,\ud x$. Then the operator
\[
f \longmapsto \left[v_\infty^{-\frac{4}{n-2}}(-\Delta-b)\right]^{-1}f
\]
is a bounded linear compact symmetric operator mapping $\mathcal{L}^2(\Omega)$ into itself. Using the spectral theorem, there exists a sequence of $H^1_0(\Omega)$ functions $\{\phi_l: l\in\mathbb{N}\}$ and a sequence of positive real numbers $\{\mu_l:l\in\mathbb{N}\}$ such that $0<\mu_1<\mu_2\le \mu_3 \le \cdots \to\infty$,
\[
-\Delta \phi_l -b\phi_l =\mu_l v_\infty^{\frac{4}{n-2}} \phi_l \quad \mbox{in }\om, \quad \phi_l=0 \quad \mbox{on }\pa \om,
\]
 and $\{\phi_l: l\in\mathbb{N}\}$ forms an orthonormal  basis of $\mathcal{L}^2(\Omega)$. In particular,
\[
\int_{\om} v_\infty^{\frac{4}{n-2}} \phi_i \phi_j=\begin{cases} 
&1 \quad \mbox{for }i=j, \\
&0 \quad \mbox{for }i\neq j.
\end{cases}
\] By the regularity theory of linear elliptic equations, $\phi_l\in C^{2+\frac{4}{n-2}}(\bar \om) \cap C^\infty(\om)$ for every $l$.  By the equation of $v_\infty$ and the positivity of  $v_\infty$, we know that $\mu_1=1$ and $\phi_1= v_\infty(\int_{\om} v_\infty^{\frac{2n}{n-2}})^{-1/2}$. It is easy to check that $\{\frac{1}{\sqrt{\mu_l}}\phi_l\}$ is also an orthonormal basis of $H_0^1(\om)$ with respect to the inner product \eqref{eq:inner-product}.

Let $L$ be the largest number such that 
\[
\mu_l\le \frac{n+2}{n-2} \quad \mbox{for all } l\le L.
\]
For $f\in L^p(\Omega)$, $p\ge 1$, we denote by $\Pi$ the projection operator 
\[
\Pi f 
= f-\sum_{i=1}^L \left(\int_{\om }f \phi_i\,\ud x\right) v_\infty ^{\frac{4}{n-2}} \phi_i. 
\]
It is clear that $\Pi( L^p(\om))=\{f\in L^p(\R^n): \int_\Omega f\phi_i=0, \ i=1,2,\cdots,L\}.$ Hence, $\Pi( L^p(\om))$ is a closed subspace of $L^p(\om)$, and thus, is a Banach space with the inherited $L^p$ norm.

We have several estimates regarding this projection.

\begin{lem}\label{lem:Lpprojection} For every $1\le p<\infty$, we can find a constant $C$ depending only on $n,b,\Omega,p$ and $v_\infty$ such that 
\[
\|f\|_{L^p(\om)} \le C \Big \|\Delta f+bf+\frac{n+2}{n-2} v_\infty^{\frac{4}{n-2}} f \Big\|_{L^p(\om)}+C \sup_{1\le l\le L} \Big|\int_\om v_\infty^{\frac{4}{n-2}} \phi_l f \Big|
\]
for all $f\in W^{2,p}(\Omega)\cap W^{1,p}_0(\Omega)$.
\end{lem}

\begin{proof} 
Suppose that this is not true. Then there exists a sequence of functions $f_k\in W^{2,p}(\Omega)\cap W^{1,p}_0(\Omega)$ such that $\|f_k\|_{L^p(\om)}=1$ for all $k$, and 
\[
\lim_{k\to\infty}\Big \|\Delta f_k+bf_k+\frac{n+2}{n-2} v_\infty^{\frac{4}{n-2}} f_k \Big\|_{L^p(\om)}+ \lim_{k\to\infty}\sup_{1\le l\le L} \Big|\int_\om v_\infty^{\frac{4}{n-2}} \phi_l f_k \Big|=0.
\]
If $p>1$, then by the $W^{2,p}$ estimates, we have $\|f_k\|_{W^{2,p}(\om)}\le C$. If $p=1$, by the estimates of Brezis-Strauss \cite{BrezisS}, $\|f_k\|_{W^{1,q}(\om)}\le C$ for some $q>1$. Therefore, by the compactness, we obtain an $f$ such that $\|f\|_{L^p(\om)}=1$,  
$\int_\om v_\infty^{\frac{4}{n-2}} \phi_l f_k =0$ for all $1\le l\le L$, and
\[
\Delta f+bf+\frac{n+2}{n-2} v_\infty^{\frac{4}{n-2}} f=0
\]
in the distribution sense. Multiplying $\phi_l$ and integrating by parts, we have
\[
\left(\mu_l- \frac{n+2}{n-2} \right)\int_\om v_\infty^{\frac{4}{n-2}} \phi_l f_k =0.
\]
Hence, $\int_\om v_\infty^{\frac{4}{n-2}} \phi_l f_k =0$ for all $l> L$. 
Meanwhile, from the elliptic regularity, we know that $f\in L^\infty(\Omega)$. Hence, $f\in\mathcal{L}^2(\Omega)$, and thus, $f\equiv 0$, which is a contradiction.
\end{proof} 

\begin{lem} \label{lem:projectionest1}  There exists a constant $C$ depending only on $n,b,\Omega,p$ and $v_\infty$ such that 
\begin{itemize}
\item[(i)] 
\[
\|f\|_{L^{\frac{n+2}{n-2}}(\om)}\le C \Big \| \Pi(\Delta f+bf+\frac{n+2}{n-2} v_\infty^{\frac{4}{n-2}} f )\Big\|_{L^{\frac{n(n+2)}{n^2+4}}(\om)}+C \sup_{1\le l\le L} \Big|\int_\om v_\infty^{\frac{4}{n-2}} \phi_l f \Big| 
\]
for all $f\in W^{2,\frac{n(n+2)}{n^2+4}}(\Omega)\cap W^{1,\frac{n(n+2)}{n^2+4}}_0(\Omega)$.
\item[(ii)]
\[
\|f\|_{L^{1}(\om)}\le C \Big \| \Pi(\Delta f+bf+\frac{n+2}{n-2} v_\infty^{\frac{4}{n-2}} f )\Big\|_{L^{1}(\om)}+C \sup_{1\le l\le L} \Big|\int_\om v_\infty^{\frac{4}{n-2}} \phi_l f \Big|
\]
for all $f\in W^{2,1}(\Omega)\cap W^{1,1}_0(\Omega)$. 
\end{itemize}

\end{lem}

\begin{proof} 
Given Lemma \ref{lem:Lpprojection}, the proof is the same as that of Lemma 6.3 in \cite{Br05}. We include it for reader's convenience. By the definition of $\Pi$, we have for $f\in W^{2,p}(\Omega)\cap W^{1,p}_0(\Omega)$ that
\begin{align*}
&\Delta f+bf+\frac{n+2}{n-2} v_\infty^{\frac{4}{n-2}} f \\
&= \Pi(\Delta f+bf+\frac{n+2}{n-2} v_\infty^{\frac{4}{n-2}} f )+\sum_{i=1}^L \left(\frac{n+2}{n-2}-\mu_i\right) \left(\int_{\om }f \phi_iv_\infty ^{\frac{4}{n-2}}\,\ud x\right) v_\infty ^{\frac{4}{n-2}} \phi_i. 
\end{align*}
Hence,
\[
\|\Delta f+bf+\frac{n+2}{n-2} v_\infty^{\frac{4}{n-2}} f \|_{L^p(\om)}\le \|\Pi(\Delta f+bf+\frac{n+2}{n-2} v_\infty^{\frac{4}{n-2}} f )\|_{L^p(\om)}+C \sup_{1\le l\le L} \Big|\int_\om v_\infty^{\frac{4}{n-2}} \phi_l f \Big|.
\]
The assertion (ii) follows from the above inequality with $p=1$ and Lemma \ref{lem:Lpprojection}.

For the assertion (i), by choosing $p=\frac{n(n+2)}{n^2+4}$ in the above inequality and using Lemma \ref{lem:Lpprojection}, we have 
\begin{align*}
&\|\Delta f+bf+\frac{n+2}{n-2} v_\infty^{\frac{4}{n-2}} f \|_{L^\frac{n(n+2)}{n^2+4}(\om)}\\
&\le \|\Pi(\Delta f+bf+\frac{n+2}{n-2} v_\infty^{\frac{4}{n-2}} f )\|_{L^\frac{n(n+2)}{n^2+4}(\om)}+C \sup_{1\le l\le L} \Big|\int_\om v_\infty^{\frac{4}{n-2}} \phi_l f \Big|
\end{align*}
and
\[
\|f\|_{L^\frac{n(n+2)}{n^2+4}(\om)} \le C \Big \|\Delta f+bf+\frac{n+2}{n-2} v_\infty^{\frac{4}{n-2}} f \Big\|_{L^\frac{n(n+2)}{n^2+4}(\om)}+C \sup_{1\le l\le L} \Big|\int_\om v_\infty^{\frac{4}{n-2}} \phi_l f \Big|.
\]
By the $W^{2,p}$ regularity theory for the Laplace equation and the Sobolev embedding $W^{2,\frac{n(n+2)}{n^2+4}}\hookrightarrow L^{\frac{n+2}{n-2}}$, we have
\[
\|f\|_{L^{\frac{n+2}{n-2}}(\om)}\le C \|f\|_{W^{2,\frac{n(n+2)}{n^2+4}}(\Omega)}\le C \Big \| \Delta f+bf+\frac{n+2}{n-2} v_\infty^{\frac{4}{n-2}} f \Big\|_{L^{\frac{n(n+2)}{n^2+4}}(\om)}+C \Big \| f \Big\|_{L^{\frac{n(n+2)}{n^2+4}}(\om)}.
\]
Then the assertion (i) is followed by combining these three inequalities.
\end{proof}

\begin{lem} \label{lem:ift} There exists $\delta_1>0$ such that for every $z=(z_1,\dots,z_L)\in \R^L$ with $|z|\le \delta_1$, there exists $\xi_z \in C_0^{\frac{3n-2}{n-2}}(\overline \om)$ satisfying $1/2\le \xi_z/v_\infty\le 2$ in $\om$, 
\[
\int_{\om} v_\infty^{\frac{4}{n-2}} (\xi_z- v_\infty) \phi_l\,\ud x= z_l, \quad l=1,\dots, L,
\]
and 
\begin{equation}\label{eq:projection0}
\Pi(\Delta \xi_z+b\xi_z+ \xi_z ^{\frac{n+2}{n-2}})=0. 
\end{equation}
Furthermore, the map $z\mapsto \xi_z$ is real analytic and $\frac{\pa }{\pa z_1}\xi_z(0)=v_\infty$, $\frac{\pa }{\pa z_l}\xi_z(0)=\phi_l$ for $2\le l\le L$.
\end{lem}

\begin{proof} Let $\xi_{z}= (1+z_1) v_\infty +\sum_{l=2}^L z_l \phi_l+h$, where $$h\in \mathcal{H}:=span\{\phi_1,\dots, \phi_L\}^{\perp}$$ and the ``$\perp$" is with respect to the inner product \eqref{eq:inner-product}. By a direct computation, 
\begin{align*}
&\Pi(\Delta \xi_z+b\xi_z+ \xi_z ^{\frac{n+2}{n-2}}) \\&= (\Delta +b)\xi_z - \sum_{l=1}^L\Big( \int_\Omega (\Delta +b)\xi_z \phi_l\Big) v_\infty^{\frac{4}{n-2}} \phi_l
+ \xi_z^{\frac{n+2}{n-2}}- \sum_{l=1}^L\Big( \int_\Omega \xi_z^{\frac{n+2}{n-2}} \phi_l\Big) v_\infty^{\frac{4}{n-2}} \phi_l\\&
= (\Delta+b)h +\xi_z^{\frac{n+2}{n-2}}- \sum_{l=1}^L\Big( \int_\Omega \xi_z^{\frac{n+2}{n-2}} \phi_l\Big) v_\infty^{\frac{4}{n-2}} \phi_l=:G(z,h).
\end{align*}
For any $p>n$, we claim that there exists a small constant $\delta>0$ such that 
\[
G: \{|z|<\delta \}\times \{\|h\|_{W^{2,p}(\om) \cap W^{1,p}_0 (\om)}<\delta\} \to \Pi( L^p(\om))
\]
is analytic. Indeed, let $\Phi(z,h)=\xi_{z}$, $\mathcal{L}u= \Delta u +bu+u^{\frac{n+2}{n-2}}$. Then we have $G=\Pi \circ \mathcal{L} \circ \Phi $. Obviously, the linear maps $\Phi$ and $\Pi$ are analytic. By Lemma 5.3 of Feireisl-Simondon \cite{FS}, $\mathcal{L}$ is also analytic in some small neighborhood of $v_\infty$ in $W^{2,p}(\om) \cap W^{1,p}_0 (\om)$.  

Note that $G(0,0)=0$ and 
\[
G_{h}(0,0)\varphi= (\Delta+b) \varphi +\frac{n+2}{n-2}\left( v_\infty^{\frac{4}{n-2}} \varphi -\sum_{l=1}^L \Big(\int_\Omega   v_\infty^{\frac{4}{n-2}} \varphi \phi_l\Big)  v_\infty^{\frac{4}{n-2}} \phi_l\right).
\]
Since $G_h(0,0)$ is coercive on $H^1_0(\om)\cap\mathcal{H}$, then $G_h(0,0): W^{2,p}(\om) \cap W^{1,p}_0 (\om)\cap\mathcal{H} \to  \Pi( L^p(\om))$ is invertible, and both $G_h(0,0)$ and $(G_h(0,0))^{-1}$ are continuous. By the Implicit Function Theorem we can find $h(z)\in W^{2,p}(\om) \cap W^{1,p}_0 (\om)\cap\mathcal{H}$ such that $G(z,h(z))=0$ and $h$ is analytic in $z$, see, e.g., Section 3.3B of Berger \cite{Berger}.  The regularity of $h(z)(\cdot)$ follows from elliptic regularity theory for the linear elliptic equation $G(z,h)=0$ in $\om$ and $h=0$ on $\pa \om$.  Since $h(0)=0$, and $0=G_z(0,0)+G_h(0,0)\pa_z h(0)= G_h(0,0)\pa_z h(0)$, we have $\pa_z h(0)=0$. It follows that $\frac{\pa }{\pa z_1}\xi_z(0)=v_\infty$ and $\frac{\pa }{\pa z_l}\xi_z(0)=\phi_l$ for $2\le l\le L$.  Therefore, the proof is completed. 
\end{proof}

The difference of the energy at $\xi_z$ and $v_\infty$ can be controlled as follows.
\begin{lem} \label{lem:lojasiewiczgradient} 
There exists a real number $\gamma\in(0,1)$ depending only on $n,b,\Omega$ and $v_\infty$ such that 
\[
F(\xi_z)-F(v_\infty) \le 2 \sup_{1\le l\le L} \Big|  \int_{\om} (\Delta \xi_z+b\xi_z +\xi_z^{\frac{n+2}{n-2}})  \phi_l\,\ud x \Big|^{1+\gamma}
\]
if $z$ is sufficiently small. 
\end{lem}

\begin{proof} 
Since $z\mapsto \xi_z$ is real analytic by Lemma \ref{lem:ift}, and $F(\cdot)$ is also real analytic by Lemma 5.3 of \cite{FS}, then the function $z\mapsto F(\xi_z)$ is real analytic. Using the {\L}ojasiewicz inequality (see Th\'eor\`eme 4 of \cite{Lo1} or Proposition 1 of \cite{Lo2} on page 92), we have  
\[
|F(\xi_z)-F(v_\infty)| \le \sup_{l} \Big|\frac{\pa }{\pa z_l} F(\xi_z)\Big|^{1+\gamma}
\]
if $z$ is sufficiently small, where $\gamma\in(0,1)$ depends only on $n,b,\Omega$ and $v_\infty$, but is not explicit. By a direct computation, we have 
\begin{align*}
\frac{\pa }{\pa z_l} F(\xi_z) =-2 \int_{\om} (\Delta \xi_z+b\xi_z +\xi_z^{\frac{n+2}{n-2}})  \frac{\pa \xi_z }{\pa z_l}\,\ud x=-2  \int_{\om} (\Delta \xi_z+b\xi_z +\xi_z^{\frac{n+2}{n-2}})  \phi_l\,\ud x. 
\end{align*}
Therefore, the proof is completed. 
\end{proof}

For every $\nu$, as in the beginning of Section \ref{sec: zero}, let $\mathcal{A}_\nu$ be the closed set of all $m$-tuplets $(x_k,\lda_k, \al_k)_{1\le k\le m}$ satisfying $(x_k,\lda_k, \al_k) \in \overline B_{\frac{1}{\lda_{k,\nu}^*}}(x_{k,\nu}^*) \times [\frac{\lda_{k,\nu}^*}{2}, \frac{3\lda_{k,\nu}^*}{2} ] \times [\frac{1}{2},\frac32]$.  Let  $\delta_1>0$ be the constant in Lemma \ref{lem:ift} and $\overline B_{\delta_1}^L$ is the open ball in $\R^L$ centered at origin with radius $\delta_1$. 
Choose an element  $(z_\nu, (x_{k,\nu},\lda_{k,\nu}, \al_{k,\nu})_{1\le k\le m}) \in  \overline B_{\delta_1}^L\times \mathcal{A}_\nu$ such that  
\be\label{eq:finite-minimum-2}
\left\|v_\nu- \xi_{z_\nu}-\sum_{k=0}^m \al_{k,\nu} \xi_{x_{k,\nu}, \lda_{k,\nu}}\right\| =\inf_{(z, (x_k,\lda_k, \al_k)_{1\le k\le m}) \in \overline B_{\delta_1}^L\times \mathcal{A}_\nu } \left\|v_\nu- \xi_z-\sum_{k=0}^m \al_{k} \xi_{x_{k}, \lda_{k}}\right\|.
\ee
Similar to \eqref{eq:bubbles-1} - \eqref{eq:bubbles-3}, we have 
\be\label{eq:bubbles-1'}
\frac{ \lda_{i, \nu} }{\lda_{j, \nu}}+\frac{ \lda_{j, \nu} }{\lda_{i, \nu}}+\lda_{i, \nu} \lda_{j, \nu} |x_{i,\nu}-x_{j,\nu} |^2 \to \infty,
\ee
and for all $k$
\be\label{eq:bubbles-2'}
\lda_{k, \nu} d(x_{k,\nu}) \to \infty 
\ee
as $\nu \to \infty$. 
In addition, $d(x_{k,\nu})>\delta/2$, and 
\be\label{eq:bubbles-3'}
\left\|v_\nu-\xi_{z_\nu}-\sum_{k=1}^m  \al_k\xi_{x_{k,\nu}, \lda_{k,\nu}}\right\| \to 0
\ee
as $\nu\to \infty$.  

By the triangle inequality, 
\begin{align*}
&\left\| \xi_{z_\nu} -v_\infty+\sum_{k=1}^m  \al_k\xi_{x_{k,\nu}, \lda_{k,\nu}}- \sum_{k=1}^m \xi_{x^*_{k,\nu}, \lda_{k,\nu}^*} \right\| \\& \le \left\|v_\nu- \xi_{z_\nu}-\sum_{k=1}^m  \al_k\xi_{x_{k,\nu}, \lda_{k,\nu}} \right\|+\left\|v_\nu-v_\infty- \sum_{k=1}^m \xi_{x^*_{k,\nu}, \lda_{k,\nu}^*} \right\|=o(1).
\end{align*}
It follows that, for all $1\le k\le m$,  
\be\label{eq:bubbles-4'}
|z_{\nu}|=o(1), \quad |x_{k,\nu}- x_{k,\nu}^{*}|=o(1) \frac{1}{\lda_{k,\nu}^*},\quad  \frac{\lda_{k,\nu}}{\lda_{k,\nu}^*}=1+o(1), \quad \al_{k,\nu}=1+o(1).  
\ee
In particular, $(z_\nu, (x_{k,\nu},\lda_{k,\nu}, \al_{k,\nu})_{1\le k\le m}) \in  \overline B_{\delta_1}^L\times \mathcal{A}_\nu$ is an interior point.  

In the sequel, we assume 
\be \label{eq:order'}
\lda_{1,\nu}\ge \lda_{2,\nu} \ge \dots \ge \lda_{m,\nu}.
\ee
 Let 
\be \label{eq:decomposition-1'}
U_\nu=\xi_{z_\nu}+ \sum_{k=1}^m  \al_{k,\nu}\xi_{x_{k,\nu}, \lda_{k,\nu}}, \quad w_\nu=v_{\nu}-U_{\nu}. 
\ee

\begin{lem} \label{lem:proj-2}  We have for $1\le l \le L$, 
\be\label{eq:proj-2}
\Big|\int_\om v_\infty^{\frac{4}{n-2}} \phi_l w_\nu\,\ud x \Big|\le o(1) \int_{\om} |w_\nu|\,\ud x,
\ee 
and for $1\le k\le m$,  
\be\label{eq:proj-3}
\begin{split}
&\Big|\int_\om \xi_{(x_{k,\nu}, \lda_{k,\nu})}^{\frac{n+2}{n-2}}  w_\nu \,\ud x\Big|+\Big|\int_\om \xi_{(x_{k,\nu}, \lda_{k,\nu})}^{\frac{n+2}{n-2}}  \frac{1-\lda^2|x-x_{k,\nu}|^2}{1+\lda^2|x-x_{k,\nu}|^2} w_\nu \,\ud x\Big |\\[2mm] &\quad  +\Big|\int_\om \xi_{(x_{k,\nu}, \lda_{k,\nu})}^{\frac{n+2}{n-2}}  \frac{\lda^2(x-x_{k,\nu})}{1+\lda^2|x-x_{k,\nu}|^2} w_\nu \,\ud x\Big | \le o(1) \Big(\int_{\om} |w_\nu|^{\frac{2n}{n-2}}\,\ud x\Big)^{\frac{n-2}{2n}}.
\end{split}
\ee
\end{lem}
\begin{proof} Let $\tilde \phi_l=\frac{\pa }{\pa z_l} \xi_z$. By \eqref{eq:bubbles-4'}, we have $\|\tilde \phi_1 -v_\infty\|_{C^2(\om)}=o(1)$ and $\|\tilde \phi_l -\phi_l\|_{C^2(\om)}=o(1)$ for $l=2,\dots, L$.  By the definition of $(z_\nu, (x_{k,\nu},\lda_{k,\nu}, \al_{k,\nu})_{1\le k\le m})$, we have 
\[
\int \nabla \tilde \phi_l \nabla w_\nu- b\tilde \phi_l w_\nu=0. 
\]
Hence,
\begin{align*}
\mu_l\int_\Omega v_\infty^\frac{4}{n-2} \phi_lw_\nu\,\ud x&=\int_\Omega \Big(-\Delta\phi_l-b\phi_l \Big)w_\nu\,\ud x\\
&=\int_\Omega \Big(\Delta(\tilde\phi_l- \phi_l)+b(\tilde\phi_l- \phi_l) \Big)w_\nu\,\ud x.
\end{align*}
Since $\mu_l>0$, then we can conclude  \eqref{eq:proj-2}. The proof of \eqref{eq:proj-3} is the same as that of Lemma \ref{lem:proj-1}.
\end{proof}

Now we can show the  non-degeneracy estimates of the second variation of $F$ for $w_\nu$.
\begin{lem} \label{lem:2nd-variational-2} For large $\nu$, we have 
\[
\frac{n+2}{n-2} \int_\om \big(v_\infty^{\frac{4}{n-2}}+ \sum_{k=1}^m  \xi_{x_{k,\nu},\lda_{k,\nu}}  ^{\frac{4}{n-2}}\big) w_\nu^2   \le (1-c) \int_{\om} (|\nabla w_\nu|^2-bw_\nu^2)\,\ud x,
\]
where $c>0$ is independent of $\nu$.
\end{lem} 

\begin{proof} We assume $w_\nu$ is not zero, otherwise there is nothing to prove. Define $\tilde w_\nu= \frac{w_\nu}{\|w_\nu\|}$. Suppose the lemma is not true.  Then we can find a subsequence of $\{\tilde w_\nu\}$ (still denoted by $\{\tilde w_\nu\}$) satisfying
\be \label{eq:coer-1b}
\lim_{\nu\to \infty}\frac{n+2}{n-2}  \int_\om  \big(v_\infty^{\frac{4}{n-2}}+ \sum_{k=1}^m  \xi_{x_{k,\nu},\lda_{k,\nu}}  ^{\frac{4}{n-2}}\big)  \tilde w_\nu^2  \ge 1. 
\ee
By \eqref{eq:sobolev-b},  
\be \label{eq:coer-2b}
\int_{\om} |\tilde w_\nu|^{\frac{2n}{n-2}} \le K_b^{\frac{n}{n-2}} \|\tilde w_\nu\|= K_b^{\frac{n}{n-2}}. 
\ee
By \eqref{eq:bubbles-1'} and \eqref{eq:order'}, we can find $R_\nu\to \infty$, $R_\nu \lda_{j,\nu}^{-1} \to 0 $ for all $1\le j\le m$, and 
\be \label{eq:coer-3b}
\frac{\lda_{i,\nu}}{R_\nu} (\lda_{j,\nu}^{-1} +|x_{i,\nu}-x_{j,\nu}|) \to \infty 
\ee
for all $i<j$. Set 
\[
\om_{j,\nu}= B_{R_\nu \lda_{j,\nu}^{-1}}(x_{j,\nu})\setminus \bigcup_{i=1}^{j-1} B_{R_\nu \lda_{i,\nu}^{-1}}(x_{i,\nu}).
\]
By \eqref{eq:coer-1b} and $\|\tilde w_\nu\|=1$, there are two cases: 
\begin{itemize} 
\item[(i).] We can find $1\le j\le m$ such that 
\[
\lim_{\nu\to \infty } \int_\om    \xi_{x_{j,\nu},\lda_{j,\nu}}  ^{\frac{4}{n-2}} \tilde w_\nu^2 >0
\]
and 
\[
\lim_{\nu\to \infty  } \int_{\om_{j,\nu}} (|\nabla \tilde w_\nu|^2- b\tilde w_\nu^2 ) \le \frac{n+2}{n-2}  \int_\om    \xi_{x_{j,\nu},\lda_{j,\nu}}  ^{\frac{4}{n-2}} \tilde w_\nu^2. 
\]
\item[(ii).] 
\[
\lim_{\nu\to \infty } \int_\om    v_\infty^{\frac{4}{n-2}} \tilde w_\nu^2 >0
\]
and 
\[
\lim_{\nu\to \infty  } \int_{\om\setminus \cup_{j}\om_{j,\nu}} (|\nabla \tilde w_\nu|^2- b\tilde w_\nu^2 ) \le \frac{n+2}{n-2}  \int_\om    v_\infty  ^{\frac{4}{n-2}} \tilde w_\nu^2. 
\]
\end{itemize}

In the first case, we can obtain a contradiction similar to that in the proof of Lemma \ref{lem:2nd-variational}. 

In the latter case, after passing to subsequence we suppose $ \tilde w_\nu \rightharpoonup \tilde w$ in $H_0^1$ as $\nu\to \infty$. It follows that 
\be \label{eq:non-trivial}
\int_\om    v_\infty^{\frac{4}{n-2}} \tilde w^2 >0
\ee
and 
\be \label{eq:lower-eigen-space}
 \int_{\om} (|\nabla \tilde w|^2- b\tilde w^2 ) \le \frac{n+2}{n-2}  \int_\om    v_\infty  ^{\frac{4}{n-2}} \tilde w^2. 
\ee
By \eqref{eq:proj-2}, we further have 
\be\label{eq:o-1}
\int_\om v_\infty^{\frac{4}{n-2}} \tilde w \phi_l=0 \quad \mbox{for }l=1,\dots, L. 
\ee
Combining \eqref{eq:lower-eigen-space} and \eqref{eq:o-1}, $\tilde w$ has to be identically zero, which contradicts \eqref{eq:non-trivial}. 

Therefore, Lemma \ref{lem:2nd-variational-2} is proved. 
\end{proof}

\begin{cor}\label{cor:coercive2} 
For large $\nu$, we have 
\[
\frac{n+2}{n-2} \int_\om U_\nu ^{\frac{4}{n-2}} w_\nu^2   \le (1-c) \int_{\om} (|\nabla w_\nu|^2-bw_\nu^2)\,\ud x,
\]
where $c>0$ is independent of $\nu$.
\end{cor}
\begin{proof}
It follows from Lemma \ref{lem:2nd-variational-2}, H\"older's inequality, the Sobolev inequality \eqref{eq:sobolev-b} and the fact that
\[
\int_\Omega \Big|U_\nu ^{\frac{4}{n-2}} -v_\infty^{\frac{4}{n-2}}- \sum_{k=1}^m  \xi_{x_{k,\nu},\lda_{k,\nu}}  ^{\frac{4}{n-2}}\Big|^{\frac{n}{2}}=o(1).
\]
\end{proof}

The following two lemmas are estimates of $v_\nu-\xi_{z_\nu}$ in $L^{\frac{n+2}{n-2}}(\om)$ and $L^{1}(\om)$, respectively.
\begin{lem}\label{lem:n+2norm} 
For large $\nu$, we have 
\[
\|v_\nu-\xi_{z_\nu}\|_{L^{\frac{n+2}{n-2}}(\om)}^{\frac{n+2}{n-2}} \le C \|v_\nu^{\frac{n+2}{n-2}}(\mathcal{R}(t_\nu)-1)\|_{L^{\frac{2n}{n+2}}(\om)}^{\frac{n+2}{n-2}}+C  \sum_{k=1}^m \lda_{k,\nu}^{\frac{2-n}{2}},
\]
where $C>0$ is independent of $\nu$.
\end{lem}
\begin{proof}
From \eqref{eq:Q}, we have
\[
\Delta v_\nu+bv_\nu+v_\nu^{\frac{n+2}{n-2}}=(1-\mathcal{R}(t_\nu)) v_\nu^{\frac{n+2}{n-2}}.
\]
Combining with \eqref{eq:projection0}, we obtain
\be\label{eq:projcorrection}
\begin{split}
&\Pi\Big(\Delta (v_\nu-\xi_{z_\nu})+b(v_\nu-\xi_{z_\nu})+ \frac{n+2}{n-2} v_\infty^\frac{4}{n-2}(v_\nu-\xi_{z_\nu})\Big)\\
&=\Pi\Big((1-\mathcal{R}(t_\nu))v_\nu^{\frac{n+2}{n-2}} -\frac{n+2}{n-2} (\xi_{z_\nu}^{\frac{4}{n-2}}-v_\infty^\frac{4}{n-2})(v_\nu-\xi_{z_\nu})\\
&\quad\quad\quad+ \xi_{z_\nu}^{\frac{n+2}{n-2}}+ \frac{n+2}{n-2}\xi_{z_\nu}^{\frac{4}{n-2}}(v_\nu-\xi_{z_\nu})-v_\nu^{\frac{n+2}{n-2}}\Big).
\end{split}
\ee
Apply (i) of Lemma \ref{lem:projectionest1}  to $v_\nu-\xi_{z_\nu}$, we obtain
\begin{align*}
&\left\|v_\nu-\xi_{z_\nu}\right\|_{L^\frac{n+2}{n-2}(\Omega)}\\
&\le C\left\|(1-\mathcal{R}(t_\nu))v_\nu^{\frac{n+2}{n-2}}\right\|_{L^\frac{n(n+2)}{n^2+4}(\Omega)} +C\left\|(\xi_{z_\nu}^{\frac{4}{n-2}}-v_\infty^\frac{4}{n-2})(v_\nu-\xi_{z_\nu})\right\|_{L^\frac{n(n+2)}{n^2+4}(\Omega)}\\
&\quad+ C\left\|\xi_{z_\nu}^{\frac{n+2}{n-2}}+ \frac{n+2}{n-2}\xi_{z_\nu}^{\frac{4}{n-2}}(v_\nu-\xi_{z_\nu})-v_\nu^{\frac{n+2}{n-2}}\right\|_{L^\frac{n(n+2)}{n^2+4}(\Omega)}+C\sup_{1\le l\le L}\left|\int_\om v_\infty^{\frac{4}{n-2}} \phi_l (v_\nu-\xi_{z_\nu})\right|.
\end{align*}
Using the estimates for all $a,b\ge 0$ that
\be\label{eq:pointwise3}
\left|a^{\frac{n+2}{n-2}} +\frac{n+2}{n-2} a^\frac{4}{n-2}(b-a)-b^\frac{n+2}{n-2}\right|\le C a^{\max(0,\frac{4}{n-2}-1)}|b-a|^{\min(\frac{n+2}{n-2},2)}+ C |b-a|^\frac{n+2}{n-2},
\ee
we obtain
\begin{align*}
&\left\|\xi_{z_\nu}^{\frac{n+2}{n-2}}+ \frac{n+2}{n-2}\xi_{z_\nu}^{\frac{4}{n-2}}(v_\nu-\xi_{z_\nu})-v_\nu^{\frac{n+2}{n-2}}\right\|_{L^\frac{n(n+2)}{n^2+4}(\Omega)}\\
&\le C \left\| |v_\nu-\xi_{z_\nu}|^{\min(\frac{n+2}{n-2},2)}+  |v_\nu-\xi_{z_\nu}|^\frac{n+2}{n-2}\right\|_{L^\frac{n(n+2)}{n^2+4}(\Omega)}\\
&\le C \left\| |v_\nu-\xi_{z_\nu}|^{\min(\frac{n+2}{n-2},2)}+  |v_\nu-\xi_{z_\nu}|^\frac{n+2}{n-2}\right\|_{L^\frac{n(n+2)}{n^2+4}(\cup_{k=1}^m B_{N/\lambda_{k,\nu}}(x_{k,\nu}))}\\
&\quad +C \left\| |v_\nu-\xi_{z_\nu}|^{\min(\frac{n+2}{n-2},2)}+  |v_\nu-\xi_{z_\nu}|^\frac{n+2}{n-2}\right\|_{L^\frac{n(n+2)}{n^2+4}(\Omega\setminus \cup_{k=1}^m B_{N/\lambda_{k,\nu}}(x_{k,\nu}))},
\end{align*}
where $N$ is a large real number to be chosen later. Using H\"older's inequality, we have
\begin{align*}
&\left\| |v_\nu-\xi_{z_\nu}|^{\min(\frac{n+2}{n-2},2)}+  |v_\nu-\xi_{z_\nu}|^\frac{n+2}{n-2}\right\|_{L^\frac{n(n+2)}{n^2+4}(\cup_{k=1}^m B_{N/\lambda_{k,\nu}}(x_{k,\nu}))}\\
&\le C \sum_{k=1}^m (N/\lambda_{k,\nu})^\frac{(n-2)^2}{2(n+2)} \left\| |v_\nu-\xi_{z_\nu}|^{\min(\frac{n+2}{n-2},2)}+  |v_\nu-\xi_{z_\nu}|^\frac{n+2}{n-2}\right\|_{L^\frac{2n}{n+2}(\Omega)}\\
& \le C \sum_{k=1}^m (N/\lambda_{k,\nu})^\frac{(n-2)^2}{2(n+2)}
\end{align*}
and
\begin{align*}
&\left\| |v_\nu-\xi_{z_\nu}|^{\min(\frac{n+2}{n-2},2)}+  |v_\nu-\xi_{z_\nu}|^\frac{n+2}{n-2}\right\|_{L^\frac{n(n+2)}{n^2+4}(\Omega\setminus \cup_{k=1}^m B_{N/\lambda_{k,\nu}}(x_{k,\nu}))}\\
&\le \left\| |v_\nu-\xi_{z_\nu}|^{\min(\frac{4}{n-2},1)}+  |v_\nu-\xi_{z_\nu}|^\frac{4}{n-2}\right\|_{L^\frac{n}{2}(\Omega\setminus \cup_{k=1}^m B_{N/\lambda_{k,\nu}}(x_{k,\nu}))}\cdot  \left\| v_\nu-\xi_{z_\nu}\right\|_{L^\frac{n+2}{n-2}(\Omega)}.
\end{align*}
Since
\begin{align*}
&\left\| v_\nu-\xi_{z_\nu}\right\|_{L^\frac{2n}{n-2}(\Omega\setminus \cup_{k=1}^m B_{N/\lambda_{k,\nu}}(x_{k,\nu}))}\\
&=  \left\|  \sum_{k=1}^m  \al_{k,\nu}\xi_{x_{k,\nu}, \lda_{k,\nu}}+ w_\nu   \right\|_{L^\frac{2n}{n-2}(\Omega\setminus \cup_{k=1}^m B_{N/\lambda_{k,\nu}}(x_{k,\nu}))}\\
&\le \sum_{k=1}^m  \al_{k,\nu}\left\|  \xi_{x_{k,\nu}, \lda_{k,\nu}} \right\|_{L^\frac{2n}{n-2}(\Omega\setminus B_{N/\lambda_{k,\nu}}(x_{k,\nu}))} + \left\|  w_\nu  \right\|_{L^\frac{2n}{n-2}(\Omega)} \\
&\le C N^{-\frac{n-2}{2}}+o(1),
\end{align*}
we have
\begin{align*}
&\left\|\xi_{z_\nu}^{\frac{n+2}{n-2}}+ \frac{n+2}{n-2}\xi_{z_\nu}^{\frac{4}{n-2}}(v_\nu-\xi_{z_\nu})-v_\nu^{\frac{n+2}{n-2}}\right\|_{L^\frac{n(n+2)}{n^2+4}(\Omega)}\\
&\le C \sum_{k=1}^m (N/\lambda_{k,\nu})^\frac{(n-2)^2}{2(n+2)} +C (N^{-\frac{n-2}{2}}+N^{-2}+o(1))\left\| v_\nu-\xi_{z_\nu}\right\|_{L^\frac{n+2}{n-2}(\Omega)}.
\end{align*}
Also,
\begin{align}
\sup_{1\le l\le L}\left|\int_\om v_\infty^{\frac{4}{n-2}} \phi_l (v_\nu-\xi_{z_\nu})\right|
&=\sup_{1\le l\le L}\left|\int_\om v_\infty^{\frac{4}{n-2}} \phi_l \Big( \sum_{k=1}^m  \al_{k,\nu}\xi_{x_{k,\nu}, \lda_{k,\nu}}+ w_\nu \Big)\right|\nonumber\\
& \le C \sum_{k=1}^m  \lda_{k,\nu}^{\frac{2-n}{2}}+ o(1)\|w_\nu\|_{L^1(\Omega)}\nonumber\\
&\le C \sum_{k=1}^m  \lda_{k,\nu}^{\frac{2-n}{2}}+o(1)\left\|v_\nu -\xi_{z_\nu}-  \sum_{k=1}^m  \al_{k,\nu}\xi_{x_{k,\nu}, \lda_{k,\nu}}\right\|_{L^1(\Omega)}\nonumber\\
&\le C \sum_{k=1}^m  \lda_{k,\nu}^{\frac{2-n}{2}}+o(1)\left\|v_\nu -\xi_{z_\nu}\right\|_{L^1(\Omega)}\label{eq:projectiondifference}.
\end{align}
Putting these facts together, we have
\begin{align*}
&\left\|v_\nu-\xi_{z_\nu}\right\|_{L^\frac{n+2}{n-2}(\Omega)}\\
&\le C\left\|(1-\mathcal{R}(t_\nu))v_\nu^{\frac{n+2}{n-2}}\right\|_{L^\frac{2n}{n+2}(\Omega)}+C \sum_{k=1}^m (N/\lambda_{k,\nu})^\frac{(n-2)^2}{2(n+2)}  \\
&+C (N^{-\frac{n-2}{2}}+N^{-2}+o(1))\left\| v_\nu-\xi_{z_\nu}\right\|_{L^\frac{n+2}{n-2}(\Omega)}+C \sum_{k=1}^m  \lda_{k,\nu}^{\frac{2-n}{2}}.
\end{align*}
By choosing $N$ sufficiently large, we obtain
\begin{align*}
\left\|v_\nu-\xi_{z_\nu}\right\|_{L^\frac{n+2}{n-2}(\Omega)}\le C\left\|(1-\mathcal{R}(t_\nu))v_\nu^{\frac{n+2}{n-2}}\right\|_{L^\frac{2n}{n+2}(\Omega)}+C \sum_{k=1}^m \lambda_{k,\nu}^{-\frac{(n-2)^2}{2(n+2)}},
\end{align*}
from which the conclusion follows.
\end{proof}

\begin{lem}\label{lem:1norm} 
For large $\nu$, we have 
\[
\|v_\nu-\xi_{z_\nu}\|_{L^{1}(\om)}\le C \|v_\nu^{\frac{n+2}{n-2}}(\mathcal{R}(t_\nu)-1)\|_{L^{\frac{2n}{n+2}}(\om)}^{\frac{n+2}{n-2}}+C  \sum_{k=1}^m \lda_{k,\nu}^{\frac{2-n}{2}},
\]
where $C>0$ is independent of $\nu$.
\end{lem}
\begin{proof}
Using \eqref{eq:projcorrection}, and applying (ii) of Lemma \ref{lem:projectionest1}  to $v_\nu-\xi_{z_\nu}$, we obtain
\begin{align*}
&\left\|v_\nu-\xi_{z_\nu}\right\|_{L^1(\Omega)}\\
&\le C\left\|(1-\mathcal{R}(t_\nu))v_\nu^{\frac{n+2}{n-2}}\right\|_{L^1(\Omega)} +C\left\|(\xi_{z_\nu}^{\frac{4}{n-2}}-v_\infty^\frac{4}{n-2})(v_\nu-\xi_{z_\nu})\right\|_{L^1(\Omega)}\\
&\quad+ C\left\|\xi_{z_\nu}^{\frac{n+2}{n-2}}+ \frac{n+2}{n-2}\xi_{z_\nu}^{\frac{4}{n-2}}(v_\nu-\xi_{z_\nu})-v_\nu^{\frac{n+2}{n-2}}\right\|_{L^1(\Omega)}+C\sup_{1\le l\le L}\left|\int_\om v_\infty^{\frac{4}{n-2}} \phi_l (v_\nu-\xi_{z_\nu})\right|.
\end{align*}
It follows from \eqref{eq:pointwise3} that
\begin{align*}
&\left\|\xi_{z_\nu}^{\frac{n+2}{n-2}}+ \frac{n+2}{n-2}\xi_{z_\nu}^{\frac{4}{n-2}}(v_\nu-\xi_{z_\nu})-v_\nu^{\frac{n+2}{n-2}}\right\|_{L^1(\Omega)}\\
&\le C \left\| |v_\nu-\xi_{z_\nu}|^{\min(\frac{n+2}{n-2},2)}+  |v_\nu-\xi_{z_\nu}|^\frac{n+2}{n-2}\right\|_{L^1(\Omega)}\\
&\le C \left\| v_\nu-\xi_{z_\nu}\right\|_{L^1(\Omega)}^{\max(0,1-\frac{n-2}{4})} \left\||v_\nu-\xi_{z_\nu}|^\frac{n+2}{n-2}\right\|_{L^1(\Omega)}^{\min(1,\frac{n-2}{4})} + C \left\| v_\nu-\xi_{z_\nu}\right\|_{L^{\frac{n+2}{n-2}}(\Omega)}^{\frac{n+2}{n-2}}\\
& \le  C \left\| v_\nu-\xi_{z_\nu}\right\|_{L^1(\Omega)}^{\max(0,1-\frac{n-2}{4})} \left\|v_\nu-\xi_{z_\nu}\right\|_{L^\frac{n+2}{n-2}(\Omega)}^{\frac{n+2}{n-2}\min(1,\frac{n-2}{4})} + C \left\| v_\nu-\xi_{z_\nu}\right\|_{L^{\frac{n+2}{n-2}}(\Omega)}^{\frac{n+2}{n-2}}\\
&\le  \frac{1}{2C} \left\| v_\nu-\xi_{z_\nu}\right\|_{L^1(\Omega)}+ C \left\| v_\nu-\xi_{z_\nu}\right\|_{L^{\frac{n+2}{n-2}}(\Omega)}^{\frac{n+2}{n-2}},
\end{align*}
where we used H\"older's inequality in the second inequality and the Young inequality in the last inequality. Combining \eqref{eq:projectiondifference}, we have
\begin{align*}
&\left\|v_\nu-\xi_{z_\nu}\right\|_{L^1(\Omega)}\\
&\le C\left\|(1-\mathcal{R}(t_\nu))v_\nu^{\frac{n+2}{n-2}}\right\|_{L^1(\Omega)}+o(1)\left\| v_\nu-\xi_{z_\nu}\right\|_{L^1(\Omega)}  \\
&+ \frac{1}{2} \left\| v_\nu-\xi_{z_\nu}\right\|_{L^1(\Omega)}+ C \left\| v_\nu-\xi_{z_\nu}\right\|_{L^{\frac{n+2}{n-2}}(\Omega)}^{\frac{n+2}{n-2}} +C \sum_{k=1}^m  \lda_{k,\nu}^{\frac{2-n}{2}}.
\end{align*}
Then the conclusion follows from Lemma \ref{lem:n+2norm}.
\end{proof}

Using the above two lemmas, we can continue to estimate $F(\xi_{z_\nu})- F(v_\infty)$ from  Lemma \ref{lem:lojasiewiczgradient}.

\begin{prop} \label{prop:regular-part}
For all large $\nu$, we have
\[
F(\xi_{z_\nu})- F(v_\infty) \le  C\left (\int_\om |\mathcal{R}(t_\nu)-1|^{\frac{2n}{n+2}} v_\nu^{\frac{2n}{n-2}}\right)^{\frac{n+2}{2n} (1+\gamma)} +C \sum_{k=1}^m \lda_{k,\nu}^{\frac{2-n}{2} (1+\gamma)},
\]
where $\gamma\in (0,1)$ is the one in Lemma \ref{lem:lojasiewiczgradient}.
\end{prop}
\begin{proof}
It follows from integration by parts that
\begin{align*}
& \int_{\om} (\Delta \xi_{z_\nu}+b\xi_{z_\nu} +\xi_{z_\nu}^{\frac{n+2}{n-2}})  \phi_l\,\ud x \\
 & =  \int_{\om} (\Delta v_\nu+bv_\nu +v_\nu^{\frac{n+2}{n-2}})  \phi_l\,\ud x + \mu_l \int_\om v_\infty^\frac{4}{n-2}\phi_l (v_\nu- \xi_{z_\nu})\,\ud x - \int_\om \phi_l (v_\nu^{\frac{n+2}{n-2}}- \xi_{z_\nu}^{\frac{n+2}{n-2}})\,\ud x\\
 & =\int_{\om} (1-\mathcal{R}(t_\nu))v_\nu^{\frac{n+2}{n-2}} \phi_l\,\ud x + \mu_l \int_\om v_\infty^\frac{4}{n-2}\phi_l (v_\nu- \xi_{z_\nu})\,\ud x - \int_\om \phi_l (v_\nu^{\frac{n+2}{n-2}}- \xi_{z_\nu}^{\frac{n+2}{n-2}})\,\ud x.
\end{align*}
Using the pointwise estimate
\[
|v_\nu^{\frac{n+2}{n-2}}- \xi_{z_\nu}^{\frac{n+2}{n-2}}| \le C \xi_{z_\nu}^{\frac{4}{n-2}} |v_\nu- \xi_{z_\nu}|+ C |v_\nu- \xi_{z_\nu}|^{\frac{n+2}{n-2}},
\]
we have
\begin{align*}
&\sup_{1\le l\le L} \Big|  \int_{\om} (\Delta \xi_z+b\xi_z +\xi_z^{\frac{n+2}{n-2}})  \phi_l\,\ud x \Big|\\
&\le C\left\|(1-\mathcal{R}(t_\nu))v_\nu^{\frac{n+2}{n-2}}\right\|_{L^{\frac{2n}{n+2}}(\Omega)} + C\left\| v_\nu-\xi_{z_\nu}\right\|_{L^1(\Omega)}+ C\left\| v_\nu-\xi_{z_\nu}\right\|_{L^\frac{n+2}{n-2}(\Omega)}^\frac{n+2}{n-2}\\
&\le C\left\|(1-\mathcal{R}(t_\nu))v_\nu^{\frac{n+2}{n-2}}\right\|_{L^{\frac{2n}{n+2}}(\Omega)} +C \left\|v_\nu^{\frac{n+2}{n-2}}(\mathcal{R}(t_\nu)-1)\right\|_{L^{\frac{2n}{n+2}}(\om)}^{\frac{n+2}{n-2}}+C  \sum_{k=1}^m \lda_{k,\nu}^{\frac{2-n}{2}}\\
&\le C\left\|(1-\mathcal{R}(t_\nu))v_\nu^{\frac{n+2}{n-2}}\right\|_{L^{\frac{2n}{n+2}}(\Omega)} +C  \sum_{k=1}^m \lda_{k,\nu}^{\frac{2-n}{2}},
\end{align*}
where we used Lemma \ref{lem:n+2norm} and Lemma \ref{lem:1norm} in the second inequality, and Proposition \ref{prop:Mq} in the last inequality. 

Then the conclusion follows from Lemma \ref{lem:lojasiewiczgradient}.
\end{proof}

\begin{cor} \label{cor:regular+singular} If $n\ge 4$ and $b>0$ satisfying \eqref{eq:b}, we have 
\[
F(U_\nu) \le F(v_\infty) +  \frac{2m}{n} Y(\mathbb{S}^n)^{n/2} + C \left(\int_\om |\mathcal{R}(t_\nu)-1|^{\frac{2n}{n+2}} v_\nu^{\frac{2n}{n-2}}\right)^{\frac{n+2}{2n} (1+\gamma)}. 
\]
\end{cor} 

\begin{proof} Let $\widetilde U_{\nu}= \sum_{k=1}^m \al_k \xi_{x_{k,\nu}, \lda_{k,\nu}} $. 
\begin{align*}
F(U_\nu)&= \int_\om |\nabla( \xi_{z_\nu}+\widetilde U_{\nu})|^2- b(\xi_{z_\nu}+\widetilde U_{\nu})^2 -\frac{n-2}{n} \int_\om (\xi_{z_\nu}+\widetilde U_{\nu})^{\frac{2n}{n-2}} \\
&= F(\xi_{z_\nu})+ F(\widetilde U_{\nu}) +2  \int_\om (\nabla  \xi_{z_\nu} \nabla \widetilde U_{\nu} -b  \xi_{z_\nu}  \widetilde U_{\nu} - \xi_{z_\nu} ^{\frac{n+2}{n-2}}  \widetilde U_{\nu})\\
&\quad- \frac{n-2}{n}\int_\om\left((\xi_{z_\nu}+ \widetilde U_{\nu})^{\frac{2n}{n-2}}-\frac{2n}{n-2}\xi_{z_\nu}^{\frac{n+2}{n-2}}\widetilde U_{\nu}- \xi_{z_\nu}^{\frac{2n}{n-2}}-\widetilde U^{\frac{2n}{n-2}} \right).
\end{align*}
We have
\begin{align*}
& \left|\int_\om (\nabla  \xi_{z_\nu} \nabla \widetilde U_{\nu} -b  \xi_{z_\nu}  \widetilde U_{\nu} - \xi_{z_\nu} ^{\frac{n+2}{n-2}}  \widetilde U_{\nu})\right|\\
 &= \left|\int_\om \Big(\Delta(\xi_{z_\nu}-v_\infty)  +b  (\xi_{z_\nu}-v_\infty)  + \xi_{z_\nu} ^{\frac{n+2}{n-2}} - v_\infty^{\frac{n+2}{n-2}} \Big) \widetilde U_{\nu})\right|\\
 & \le o(1)  \sum_{k=1}^m \lda_k^{\frac{2-n}{2}}.
 \end{align*}
By Lemma \ref{lem:basic-1}, there exists $c>0$, depending only on $n$ such that 
\begin{align*}
(\xi_{z_\nu}+ \widetilde U_{\nu})^{\frac{2n}{n-2}}-\frac{2n}{n-2}\xi_{z_\nu}^{\frac{n+2}{n-2}}\widetilde U_{\nu}- \xi_{z_\nu}^{\frac{2n}{n-2}}-\widetilde U^{\frac{2n}{n-2}}\ge
 \begin{cases}  c\xi_{z_\nu}^{\frac{4}{n-2}}(\widetilde U_{\nu})^2 ,& \quad \mbox{if } \xi_{z_\nu} 
 \ge U_{\nu}', \\ 
 c\xi_{z_\nu}(\widetilde U_{\nu})^{\frac{n+2}{n-2}} , &\quad \mbox{if } \xi_{z_\nu} < U_{\nu}'. 
 \end{cases}
\end{align*}
Since $v_\infty/2\le \xi_{z_\nu}\le 2 v_\infty$, and 
\[
\int_{|x|<\sqrt{\lda^{-1}}}  \left(\frac{\lda}{1+\lda^2 |x|^2}\right)^{\frac{n+2}{2}} \ge \lda^{-\frac{n-2}{2}} \int_{|y|<1} \left(\frac{1}{1+ |y|^2}\right)^{\frac{n+2}{2}} \ge c \lda^{-\frac{n-2}{2}}, 
\]
we have
\[
\int_\om\left((\xi_{z_\nu}+ \widetilde U_{\nu})^{\frac{2n}{n-2}}-\frac{2n}{n-2}\xi_{z_\nu}^{\frac{n+2}{n-2}}\widetilde U_{\nu}- \xi_{z_\nu}^{\frac{2n}{n-2}}-\widetilde U^{\frac{2n}{n-2}} \right)\ge c \sum_{k=1}^m \lda_k^{\frac{2-n}{2}}.
\]
Then, the conclusion follows from Proposition \ref{prop:regular-part} and Corollary \ref{cor:pure-singular}. 
\end{proof} 

\section{Convergence}\label{sec:convergence} 

Using the  estimates in Corollaries \ref{cor:coercive} and \ref{cor:coercive2}, we have the following estimate of $F(v_{\nu}) -F_\infty $ for any sequence of times $\{t_\nu: \nu\in\mathbb{N}\}$.
\begin{prop}\label{prob:subsequenceestimate}  
Let $n\ge 4$, and $b>0$ satisfy \eqref{eq:b}. Let $\{t_\nu: \nu\in\mathbb{N}\}$ be a sequence of times such that $t_\nu\to \infty$ as $\nu\to\infty$. Then, we can find a real number $\gamma\in (0,1)$ and a
constant $C>0$ such that, after passing to a subsequence, we have
\[
F(v_{\nu}) -F_\infty \le  C \left(\int_\om |\mathcal{R}(t_\nu)-1|^{\frac{2n}{n+2}} v_\nu^{\frac{2n}{n-2}}\right)^{\frac{n+2}{2n} (1+\gamma)}
\]
for all integers $\nu$ in that subsequence, where $F_\infty$ is the one defined in \eqref{eq:f_infty}. Note that $\gamma$ and $C$ may depend
on the sequence $\{t_\nu: \nu\in\mathbb{N}\}$.
\end{prop} 

\begin{proof} 
It follows from \eqref{eq:f_infty} that $F(v_\infty)=F_\infty$. Recall that $U_\nu=v_\nu-w_\nu$.  We have  
\begin{align*}
F(v_{\nu}) -F(U_\nu)&= \frac{n-2}{n}\int_\om (U_\nu^{\frac{2n}{n-2}}-v_\nu^{\frac{2n}{n-2}}) +2 \int_\om (\nabla v_\nu \nabla w_\nu -b v_\nu w_\nu)    -\int_\om (| \nabla w_\nu|^2 -b  w_\nu^2)  \\&
= \frac{n-2}{n}\int_\om (U_\nu^{\frac{2n}{n-2}}-v_\nu^{\frac{2n}{n-2}}) +2 \int_\om \mathcal{R} v_\nu^{\frac{n+2}{n-2}} w_\nu    -\int_\om (| \nabla w_\nu|^2 -b  w_\nu^2)\\&= 
\frac{n-2}{n}\int_\om \left(U_\nu^{\frac{2n}{n-2}}-v_\nu^{\frac{2n}{n-2}} +\frac{2n}{n-2} v_\nu^{\frac{n+2}{n-2}} w_\nu -\frac{n(n+2)}{(n-2)^2} U_\nu^{\frac{4}{n-2}} w_\nu^2\right) \\& \quad  +2 \int_\om (\mathcal{R}-1)v_\nu^{\frac{n+2}{n-2}} w_\nu -\int_\om \left(| \nabla w_\nu|^2 -b  w_\nu^2 -\frac{n+2}{n-2} U_\nu^{\frac{4}{n-2}} w_\nu^2\right)
\end{align*}
Using the pointwise estimate
\begin{align*}
 &\left|U_\nu^{\frac{2n}{n-2}}-v_\nu^{\frac{2n}{n-2}} +\frac{2n}{n-2} v_\nu^{\frac{n+2}{n-2}} w_\nu -\frac{n(n+2)}{(n-2)^2} U_\nu^{\frac{4}{n-2}} w_\nu^2\right|\\
&=\left|U_\nu^{\frac{2n}{n-2}}- (w_\nu+U_\nu)^{\frac{2n}{n-2}}+\frac{2n}{n-2}(w_\nu+U_\nu)^{\frac{n+2}{n-2}} w_\nu -\frac{n(n+2)}{(n-2)^2} U_\nu^{\frac{4}{n-2}} w_\nu^2\right|    
 \\& \le C U_\nu^{\max\{0,\frac{4}{N-2}-1\}} |w_\nu|^{\min\{\frac{2N}{N-2},3\}} +C |w_\nu |^{\frac{2N}{N-2}},
 \end{align*}
 it follows that 
 \begin{align*}
 &\int_\om \left|U_\nu^{\frac{2n}{n-2}}-v_\nu^{\frac{2n}{n-2}} +\frac{2n}{n-2} v_\nu^{\frac{n+2}{n-2}} w_\nu -\frac{n(n+2)}{(n-2)^2} U_\nu^{\frac{4}{n-2}} w_\nu^2\right|  \\& 
 \le C \int_\om U_\nu^{\max\{0,\frac{4}{n-2}-1\}} |w_\nu|^{\min\{\frac{2n}{n-2},3\}} +C \int_\om |w_\nu |^{\frac{2n}{n-2}} \\&
 \le C \left(\int_\om |w_\nu |^{\frac{2n}{n-2}} \right)^{\frac{n-2}{n}\min\{\frac{n}{n-2}, \frac32\}}. 
 \end{align*}
 By H\"older inequality and Cauchy inequality, we have 
 \begin{align*}
\left| \int_\om (\mathcal{R}-1)v_\nu^{\frac{n+2}{n-2}} w_\nu \right| & \le C\left(\int_\om |\mathcal{R}-1|^{\frac{2n}{n+2}} v_\nu^{\frac{2n}{n-2}}\right)^{\frac{n+2}{2n}} \left(\int_\om |w_\nu|^{\frac{2n}{n-2}} \right)^{\frac{n-2}{2n}} \\&
\le \va \left(\int_\om |w_\nu|^{\frac{2n}{n-2}} \right)^{\frac{n-2}{n}} + C(\va)\left(\int_\om |\mathcal{R}-1|^{\frac{2n}{n+2}} v_\nu^{\frac{2n}{n-2}}\right)^{\frac{n+2}{n}} . 
 \end{align*}
 Finally, by Corollaries \ref{cor:coercive} and \ref{cor:coercive2}, we have  
 \[
 \int_\om |\nabla w_\nu|^2 -b w_\nu^2- \frac{n+2}{n-2} U_\nu^{\frac{4}{n-2}} w_\nu^2\ge 
 \frac{1}{C} \left(\int_\om |w_\nu|^{\frac{2n}{n-2}} \right)^{\frac{n-2}{n}}. 
 \]
 Since $\displaystyle\int_\om |w_\nu|^{\frac{2n}{n-2}}  \to 0$, we have, by choosing $\va $ being small, that 
 \begin{align*}
 F(v_{\nu}) -F(U_\nu) &\le C \left(\int |w_\nu |^{\frac{2n}{n-2}} \right)^{\frac{n-2}{n}\min\{\frac{n}{n-2}, \frac32\}} +C\left(\int |\mathcal{R}-1|^{\frac{2n}{n+2}} v_\nu^{\frac{2n}{n-2}}\right)^{\frac{n+2}{n}}  \\& \quad - \frac{1}{2C} \left(\int |w_\nu|^{\frac{2n}{n-2}} \right)^{\frac{n-2}{n}}\\&
 \le C\left(\int |\mathcal{R}-1|^{\frac{2n}{n+2}} v_\nu^{\frac{2n}{n-2}}\right)^{\frac{n+2}{n}}. 
 \end{align*}
 By Corollary \ref{cor:regular+singular}, the proof is completed. 
 \end{proof} 
 
Then we can show the estimate for all large time. 
\begin{cor}\label{cor:fullestimate}
There exist real numbers $\gamma\in (0,1)$ and $t_0>0$ such that
\[
F(v(t)) -F_\infty \le  \left(\int_\om |\mathcal{R}-1|^{\frac{2n}{n+2}} v(x,t)^{\frac{2n}{n-2}}\,\ud x\right)^{\frac{n+2}{2n} (1+\gamma)}
\]
for all $t\ge t_0$.
\end{cor}
\begin{proof}
Suppose this is not true. Then, there exists a sequence of
times $\{t_\nu: \nu\in\mathbb{N}\}$ such that $t_\nu>\nu$ and 
\[
F(v(t_\nu)) -F_\infty \ge  \left(\int_\om |\mathcal{R}(t_\nu)-1|^{\frac{2n}{n+2}} v(x,t_\nu)^{\frac{2n}{n-2}}\,\ud x\right)^{\frac{n+2}{2n} (1+\frac{1}{\nu})}
\]
for all $\nu\in\mathbb{N}$. By applying Proposition \ref{prob:subsequenceestimate} to this sequence $\{t_\nu: \nu\in\mathbb{N}\}$, there exists an infinite subset $I\subset\mathbb{N}$, a real number
$\alpha\in (0,1)$ and $C>0$ such that
\[
F(v(t_\nu)) -F_\infty \le  C \left(\int_\om |\mathcal{R}(t_\nu)-1|^{\frac{2n}{n+2}} v(x,t_\nu)^{\frac{2n}{n-2}}\,\ud x\right)^{\frac{n+2}{2n} (1+\alpha)}
\]
for all $\nu\in I$. Thus, we have
\[
1\le C \left(\int_\om |\mathcal{R}(t_\nu)-1|^{\frac{2n}{n+2}} v(x,t_\nu)^{\frac{2n}{n-2}}\,\ud x\right)^{\frac{n+2}{2n} (\alpha-\frac{1}{\nu})}
\]
for all $\nu\in I$. However, from Proposition \ref{prop:Mq}, we have
\[
\lim_{\nu\to\infty}\left(\int_\om |\mathcal{R}(t_\nu)-1|^{\frac{2n}{n+2}} v(x,t_\nu)^{\frac{2n}{n-2}}\,\ud x\right)^{\frac{n+2}{2n} (\alpha-\frac{1}{\nu})}=0.
\]
We have reached a contradiction. 
\end{proof}

Now we can use a differential inequality of $F$ to obtain a decay estimate.
\begin{prop} \label{prop:converge-1}
There exist $\theta>0$ and $C>0$ such that for all $T>1$, there holds
\[
\int_{T}^\infty M_2(t)^{1/2} \,\ud t \le C T^{-\theta},
\]
where $M_2$ is defined in \eqref{eq:definitionMq}.
\end{prop} 

\begin{proof} 
It follows from  Corollary \ref{cor:fullestimate}, H\"older's inequality, and \eqref{eq:BN-5} that 
\[
0\le F(v(t))-F_\infty \le \left(\int_\om |\mathcal{R}-1|^{\frac{2n}{n+2}} v(x,t)^{\frac{2n}{n-2}}\,\ud x\right)^{\frac{n+2}{2n} (1+\gamma)}  \le C M_2(t)^{\frac{1+\gamma}{2}}. 
\]
It follows that 
\begin{align*}
\frac{\ud }{\ud t}(F(v(t))-F_\infty )= -\frac{2(n-2)}{n+2} M_2(t) \le -C (F(v(t))-F_\infty)^{\frac{2}{1+\gamma}}. 
\end{align*}
Hence, 
\[
\frac{\ud }{\ud t}(F(v(t))-F_\infty  )^{\frac{\gamma-1}{\gamma+1}} \ge C \frac{1-\gamma}{1+\gamma}>0. 
\] 
It follows that  
\[
F(v(t))-F_\infty  \le Ct^{-\frac{1+\gamma}{1-\gamma}}
\]
for sufficiently large $t$. 
Then we have 
\begin{align*}
\left(\int_{T}^{2T} M_2(s)^{1/2}\,\ud s\right)^2&  \le T \int_{T }^{2T } M_2(s) \,\ud s  
\\& \le \frac{n+2}{n-2}  T  (F(v(T))- F(v(2T))) \\&
\le  \frac{n+2}{n-2}  T  (F(v(T))- F_\infty)  \\& 
 \le C T^{-\frac{2\gamma}{1-\gamma}},
\end{align*} 
where we used the monotonicity of $F$. It follows that 
\be\label{eq:M2rate}
\int_{T}^\infty M_2(t)^{1/2} \,\ud t  = \sum_{k=0}^{\infty}  \int_{2^kT}^{2^{k+1}T} M_2(t)^{1/2} \,\ud t \le C T^{-\frac{\gamma}{1-\gamma}}\sum_{k=1}^{\infty} 2^{-\frac{\gamma}{1-\gamma} k}\le CT^{-\frac{\gamma}{1-\gamma}}. 
\ee
This finishes the proof. 
\end{proof} 

We are ready to show the uniform boundedness, and uniform higher order estimates.

\begin{prop}\label{prop:boundedness} 
For any $\va>0$, there exists $T_0>0$ such that 
\be\label{eq:closedness}
\|v(\cdot,t)-v(\cdot, T_0)\|_{L^{\frac{2n}{n-2}}(\om)} <\va\quad\mbox{for all }t>T_0.
\ee 
Consequently, there exists $C>0$ depending only $n,b,\Omega$ and $u_0$ such that
\be\label{eq:closedness2}
v(x, t)\le C\quad\mbox{ in }\ \Omega \mbox{ for all }t>1.
\ee 
\end{prop} 

\begin{proof} 
For $b>a>1$, using the pointwise estimate
\[
 |v(x,b)-v(x,a)|^{\frac{n}{n-2}}\le |v(x,b)^{\frac{n}{n-2}} -v(x,a) ^{\frac{n}{n-2}}|,
\]
we have
\begin{equation}  \label{eq:stable-5}
\begin{split}
\left(\int_{\om} |v(x,b)-v(x,a)|^{\frac{2n}{n-2}}\,\ud x\right)^{1/2} &\le\left( \int_{\om} |v(x,b)^{\frac{n}{n-2}} -v(x,a) ^{\frac{n}{n-2}} |^2\,\ud x\right)^{1/2} \\&
\le\left(\int_{\om} \Big(\int_{a}^{b} |\pa_t (v(x,t)^{\frac{n}{n-2}})| \,\ud t\Big) ^2 \,\ud x \right)^{1/2}\\&
\le \int_{a}^{b} \left(\int_{\om} | \pa_t (v(x,t)^{\frac{n}{n-2}})| ^2 \,\ud x\right)^{1/2}\,\ud t\\&
\le C \int_{a}^{\infty} M_2(t)^{1/2} \,\ud t\\
&\le C a^{-\theta},
\end{split}
\end{equation} 
where we used Minkowski's integral inequality in the third inequality, and  Proposition \ref{prop:converge-1} in the last inequality. Hence, for any $\va>0$, there exists $T_0>0$ such that \eqref{eq:closedness} holds. 

To show the $L^\infty$ bound in \eqref{eq:closedness2}, we need to use the following Brezis-Kato \cite{BrezisKato} estimate (see also Lemma B.3 in Appendix B of Struwe \cite{St2}): there exists $\delta>0$ depending only on $n$ and $\Omega$ such that if $v\in H^1_0(\Omega)$ is a weak solution of 
\[
-\Delta v=c_1 v+ c_2 v \quad\mbox{in }\Omega,
\]
where $\|c_1\|_{L^{\frac{n}{2}}(\Omega)}\le \delta$ and $c_2\in L^p(\Omega)$ for some $p>\frac{n}{2}$, then there exist $C>0$ and $q>\frac{2n}{n-2}$ depending only on $n,\delta,p,\Omega$ and $\|c_2\|_{L^p(\Omega)}$ such that
\[
\|v\|_{L^q(\Omega)}\le C \|v\|_{L^2(\Omega)}.
\]

Let $T_0>0$ be the one in \eqref{eq:closedness} with some $\varepsilon < \delta/2$  that $(2\va)^{\frac{4}{n-2}}\le \delta/2$. By Proposition \ref{prop:Mq}, there exists $T_1>0$ such that $M(t)_{\frac{n}{2}}<\delta/2$ for all $t>T_1$. Let $T_2=\max(T_0,T_1)$,  
\[
U=\{x\in\Omega: |v(x,t)-v(x,T_2)|>\max_\Omega v(\cdot,T_2)\}
\]
 and $\chi_{U}$ be the characteristic function of $U$. Then for $t>T_2$, we have
\begin{align*}
\|v^{\frac{4}{n-2}} \chi_{U}\|_{L^{\frac{n}{2}}(\Omega)}&=(\|v\|_{L^{\frac{2n}{n-2}}(U)})^{\frac{4}{n-2}}\\
&\le (\|v(\cdot,T_2)\|_{L^{\frac{2n}{n-2}}(U)}+\|v(\cdot,t)-v(\cdot,T_2)\|_{L^{\frac{2n}{n-2}}(U)})^{\frac{4}{n-2}}\\
&\le (\max_\Omega |v(\cdot,T_2)|\cdot |U|^\frac{n-2}{2n}+\|v(\cdot,t)-v(\cdot,T_2)\|_{L^{\frac{2n}{n-2}}(U)})^{\frac{4}{n-2}}\\
&\le (\|v(\cdot,t)-v(\cdot,T_2)\|_{L^{\frac{2n}{n-2}}(U)}+\|v(\cdot,t)-v(\cdot,T_2)\|_{L^{\frac{2n}{n-2}}(U)})^{\frac{4}{n-2}}\\
&\le (2\va)^{\frac{4}{n-2}}\\
&\le \delta/2.
\end{align*}
where we used Chebyshev's inequality in the third inequality, and \eqref{eq:stable-5} in the fourth inequality. 
 From the definition of $\mathcal{R}$ in \eqref{eq:Q},  we have
\be\label{eq:ellipticsplit}
-\Delta v=bv+\mathcal{R} v^{\frac{n+2}{n-2}} =: V_1 v + V_2 v \quad \mbox{in }\om, \quad v=0 \quad \mbox{on }\pa \om,
\ee
where $V_1= (\mathcal{R}-1) v^{\frac{4}{n-2}} +  v^{\frac{4}{n-2}} \chi_{U}$ and $V_2=  (1-\chi_{U})v^{\frac{4}{n-2}}+b$.  Then $V_2\in L^\infty(\Omega)$ and
\[
\|V_1\|_{L^{\frac{n}{2}}(\Omega)}\le M(t)_{\frac{n}{2}}+\|v^{\frac{4}{n-2}} \chi_{U}\|_{L^{\frac{n}{2}}(\Omega)}\le \delta/2+\delta/2=\delta
\]
for all $t>T_2$. Then by the Brezis-Kato estimate, there exist $q>\frac{2n}{n-2}$ and $C>0$ such that
\[
\|v\|_{L^q(\Omega)}\le C \|v\|_{L^2(\Omega)}\le C,
\] 
where we used H\"older's inequality and \eqref{eq:BN-5} in the last inequality. Now $V_1$  belongs to $L^p$ for some $p>\frac{n}{2}$, and then the standard Moser's iteration will lead to \eqref{eq:closedness2}.
\end{proof}

\begin{thm}\label{thm:smoothbound}
There exists  $C>0$ depending only $n,b,\Omega$ and $u_0$ such that
\be\label{eq:smoothestimate}
\|v(\cdot,t)\|_{C^{2+\frac{n+2}{n-2}}(\overline\Omega)}\le C\quad\mbox{for all }t>1.
\ee 
\end{thm}
\begin{proof}
By using \eqref{eq:closedness2}, it follows from Proposition 6.2 in \cite{DKV} (more precisely, its proof) that
\[
\frac{1}{C} d(x)\le v(\cdot,t)\le C d(x)\quad\mbox{for all }x\in\Omega,\ t>1,
\] 
where $d(x):=\mbox{dist}(x,\partial\Omega)$. Then \eqref{eq:smoothestimate} follows from Theorem 5.1 in \cite{JX19}.
\end{proof}

Let us conclude this section with the proof of Theorem \ref{thm:critical}.

\begin{proof}[Proof of Theorem \ref{thm:critical}]
It follows from Proposition \ref{prop:boundedness},  Theorem \ref{thm:smoothbound} and \eqref{eq:BN-5}  that there exists a nonzero stationary solution $v_\infty$ of \eqref{eq:BN-3} such that
\[
\lim_{t\to\infty}\|v(\cdot,t)-v_\infty\|_{C^3(\overline\Omega)}=0.
\]
From \eqref{eq:stable-5}, we know that there exist $C>0$ and $\theta>0$ such that
\[
\|v(\cdot,t)-v_\infty\|_{L^\frac{2n}{n-2}(\Omega)}\le C t^{-\theta}\quad\mbox{for all }t>1.
\]
Using \eqref{eq:smoothestimate} and Gagliardo interpolation inequalities (see, e.g., (12)--(13) in Blanchet-Bonforte-Dolbeault-Grillo-V\'azquez \cite{BB+}), we have
\[
\|v(\cdot,t)-v_\infty\|_{C^1(\overline\Omega)}\le C t^{-\theta}\quad\mbox{for all }t>1
\]
with a possibly different $\theta$. Since $v(\cdot,t)\equiv v_\infty\equiv 0$ on $\partial\Omega$, we have for all $x\in\Omega$ that
\[
\left|\frac{v(x,t)-v_\infty(x)}{v_\infty(x)}\right| \le \frac{\|v(\cdot,t)-v_\infty\|_{C^1(\overline\Omega)} \ d(x)}{d(x)/C} \le C \|v(\cdot,t)-v_\infty\|_{C^1(\overline\Omega)} \le C t^{-\theta}\quad\mbox{for all }t>1.
\]
That is,
\[
\left\|\frac{v(\cdot,t)}{v_\infty}-1\right\|_{L^\infty(\om)}\le Ct^{-\theta}\quad\mbox{for all }t\ge 1.
\]
Using \eqref{eq:smoothestimate} again, we have
\[
\left\|\frac{v(\cdot,t)}{v_\infty}\right\|_{C^{1+\frac{n+2}{n-2}}(\overline\Omega)}\le C\quad\mbox{for all }t>1.
\]
Then by interpolation inequalities, we have
\be\label{eq:relativealgebraic}
\left\|\frac{v(\cdot,t)}{v_\infty}-1\right\|_{C^2(\overline\om)}\le Ct^{-\theta}\quad\mbox{for all }t\ge 1.
\ee

Now let us assume that $\Omega$ satisfies the condition \eqref{condition}. Then there exists $C>0$  such that for all $\varphi\in H^1_0$ satisfying
\[
-\Delta \varphi- b \varphi -\frac{n+2}{n-2}v_\infty^{\frac{4}{n-2}}\varphi=f\quad\mbox{in }\Omega,
\]
there holds
\be\label{eq:invertibleestimate}
\|\varphi\|_{H^1_0(\Omega)}\le C \|f\|_{L^2(\Omega)}.
\ee
Here, the $C$ depends only on $n,b,\Omega$ and $v_\infty$. This estimate can be proved as follows. First, it follows Theorem 6 in Section 6.2.3 of Evans \cite{Evans} that there exists $C>0$ such that $\|\varphi\|_{L^2(\Omega)}\le C \|f\|_{L^2(\Omega)}$. Secondly, multiplying $\varphi$ to its equation and integrating by part, it leads to \eqref{eq:invertibleestimate}.

On one hand, using the equation of $v_\infty$, we have
\begin{align*}
&F(v(t))-F_\infty
\\&= \int_{\om } \Big( |\nabla v(\cdot,t)|^2 -b v(\cdot,t)^2 -\frac{n-2}{n} v(\cdot,t)^{\frac{2n}{n-2}}\Big) - \int_{\om } \Big( |\nabla v_\infty|^2 -b v_\infty^2 -\frac{n-2}{n} v_\infty^{\frac{2n}{n-2}}\Big)\\
&\quad-2\int_\om \Big(\nabla v_\infty\nabla (v(\cdot,t)-v_\infty) -b v_\infty (v(\cdot,t)-v_\infty)- v_\infty^{\frac{n+2}{n-2}} (v(\cdot,t)-v_\infty)\Big)\\
&= \int_{\om }  |\nabla (v(\cdot,t)- v_\infty)|^2 -b (v(\cdot,t)- v_\infty)^2 -\frac{n-2}{n} \Big(v(\cdot,t)^{\frac{2n}{n-2}}-v_\infty^{\frac{2n}{n-2}}-\frac{2n}{n-2} (v(\cdot,t)- v_\infty)\Big)\\
&\le C \|v(\cdot,t)-v_\infty\|_{H^1_0(\Omega)}^2.
\end{align*}
On the other hand,  using the equation of $v_\infty$, we have
\begin{align*}
&\|\Delta v(\cdot, t)+b v(\cdot,t)+v(\cdot,t)^{\frac{n+2}{n-2}} \|_{L^2(\Omega)}\\
&=\|\Delta (v(\cdot, t)-v_\infty)+b (v(\cdot,t)-v_\infty)+v(\cdot,t)^{\frac{n+2}{n-2}}-v_\infty^{\frac{n+2}{n-2}} \|_{L^2(\Omega)}\\
&\ge \left\|\Delta (v(\cdot, t)-v_\infty)+b (v(\cdot,t)-v_\infty)+\frac{n+2}{n-2} v_\infty^\frac{4}{n-2}(v(\cdot,t)-v_\infty)\right\|_{L^2(\Omega)}\\
&\quad- \left\|v(\cdot,t)^{\frac{n+2}{n-2}}-v_\infty^{\frac{n+2}{n-2}}-\frac{n+2}{n-2} v_\infty^\frac{4}{n-2}(v(\cdot,t)-v_\infty)\right\|_{L^2(\Omega)}\\
&\ge \frac 1C \|v(\cdot, t)-v_\infty\|_{H^1_0(\Omega)}- C \left\| v_\infty ^{\max(0, \frac{4}{n-2}-1)}|v(\cdot,t)-v_\infty|^{\min(2, \frac{n+2}{n-2})}\right\|_{L^2(\Omega)}\\
&\ge \frac 1C \|v(\cdot, t)-v_\infty\|_{H^1_0(\Omega)} -o(1) \|v(\cdot, t)-v_\infty\|_{H^1_0(\Omega)} \quad \mbox{for large }t\\
&\ge \frac 1C \|v(\cdot, t)-v_\infty\|_{H^1_0(\Omega)}\quad \mbox{for large }t,
\end{align*}
where we used \eqref{eq:invertibleestimate} in the second inequality. The constant $C$ depends only on $n,b,\Omega$ and $u_0$. Combining these two inequalities together, we have 
\[
F(v(t))-F_\infty\le \|\Delta v(\cdot, t)+b v(\cdot,t)+v(\cdot,t)^{\frac{n+2}{n-2}} \|_{L^2(\Omega)}^2=CM_2(t).
\]
It follows that 
\begin{align*}
\frac{\ud }{\ud t}(F(v(t))-F_\infty )= -\frac{2(n-2)}{n+2} M_2(t) \le -C (F(v(t))-F_\infty),
\end{align*}
and thus,
\[
F(v(t))-F_\infty  \le Ce^{-\gamma t}
\]
for some $C>0,\gamma>0$.
Hence, the proof of \eqref{eq:M2rate} will give
\[
\int_{T}^\infty M_2(t)^{1/2} \,\ud t \le Ce^{-\gamma t}.
\]
From \eqref{eq:stable-5}, we obtain that 
\[
\|v(\cdot,t)-v_\infty\|_{L^\frac{2n}{n-2}(\Omega)} \le Ce^{-\gamma t}.
\]
Then the proof of \eqref{eq:relativealgebraic} gives
\be\label{eq:relativeexp}
\left\|\frac{v(\cdot,t)}{v_\infty}-1\right\|_{C^2(\overline\om)}\le Ce^{-\gamma t}\quad\mbox{for all }t\ge 1.
\ee
This finishes the proof of Theorem \ref{thm:critical}.
\end{proof}

\section{Subcritical case}\label{sec:subcritical}
In this last section, we consider the Sobolev subcritical case \eqref{eq:subcritical2}, and prove Theorem \ref{thm:subcritical}. The proof is similar to that of Theorem \ref{thm:critical}.
\begin{proof}[Proof of Theorem \ref{thm:subcritical}]
First, we know from Proposition 6.2 in \cite{DKV} that
\[
\frac{1}{C} d(x)\le v(\cdot,t)\le C d(x)\quad\mbox{for all }x\in\Omega,\ t>1,
\] 
Secondly, it follows from Theorem 1.1 in \cite{FS} and Theorem 5.1 in \cite{JX19}  that there exists a nonzero stationary solution $v_\infty$ of \eqref{eq:subcritical2} such that
\[
\lim_{t\to\infty}\|v(\cdot,t)-v_\infty\|_{C^3(\overline\Omega)}=0.
\]
Let  
\begin{align*}
F(v(t))&= \int_{\om}\left( |\nabla v(x,t)|^2 -\frac{2}{p+1} v(x,t)^{p+1}\right)\,\ud x, \\
\mathcal{R}&= -v^{-p} \Delta v=1-\frac{p\partial_t v}{v},\\
M_q(t)&=\int_{\om} |\mathcal{R}-1|^q v^{p+1} \,\ud x, 
\end{align*}
where $q\ge 1$. Note that 
\[
\frac{\ud }{\ud t} F(v(t))= -2 \int_\om (\Delta v+v^p) \pa_t v= -2p\int_\om |\pa_t v|^2 v^{p-1}=-\frac{2}{p }M_2(t).
\] 
Hence $F(v(\cdot,t))$ deceases to $F(v_\infty)$ as $t\to\infty$. Furthermore, 
\[
\ud  F(v)=-2\Delta v-2v^p=2(\mathcal{R}-1)v^p.
\]
Hence,
\[
\|\ud  F(v(\cdot,t))\|_{L^2(\Omega)} \le C M_{2}(t)^{1/2}. 
\]
Then it follows from Proposition 6.1 in \cite{FS} that there exist $C>0, T_0>0, \gamma>0$ such that for all $t>T_0$, we have
\[
F(v(t))-F(v_\infty)\le C\|\ud  F(v(\cdot,t))\|_{L^2(\Omega)}^{1+\gamma}\le C M_{2}(t)^{\frac{1+\gamma}{2}}. 
\]
Therefore, similar to the proof of Proposition \ref{prop:converge-1}, there exist $\theta\in (0,1)$ and $C>0$ such that for all $T>1$, there holds
\[
\int_{T}^\infty M_2(t)^{1/2} \,\ud t \le C T^{-\theta},
\]
Then it follows from the proof of \eqref{eq:stable-5} that
\be\label{eq:subM2exp}
\|v(\cdot,t)-v_\infty\|_{L^{p+1}(\om)}\le  C \int_t^\infty M_2(s)^{1/2}\,\ud s\le C t^{-\theta},
\ee
Using Theorem 5.1 in \cite{JX19} (which is the regularity estimate) and interpolation inequalities, we have
\[
\|v(\cdot,t)-v_\infty\|_{C^1(\overline\Omega)}\le C t^{-\theta}\quad\mbox{for all }t>1.
\]
Since $v(\cdot,t)\equiv v_\infty\equiv 0$ on $\partial\Omega$, we have for all $x\in\Omega$ that
\[
\left|\frac{v(x,t)-v_\infty(x)}{v_\infty(x)}\right| \le C \|v(\cdot,t)-v_\infty\|_{C^1(\overline\Omega)} \le C t^{-\theta}\quad\mbox{for all }t>1.
\]
That is,
\[
\left\|\frac{v(\cdot,t)}{v_\infty}-1\right\|_{L^\infty(\om)}\le Ct^{-\theta}\quad\mbox{for all }t\ge 1.
\]
Using Theorem 5.1 in \cite{JX19} again, we have
\[
\left\|\frac{v(\cdot,t)}{v_\infty}\right\|_{C^{1+p}(\overline\Omega)}\le C\quad\mbox{for all }t>1.
\]
Then by interpolation inequalities, we have
\be\label{eq:relativealgebraicsub} 
\left\|\frac{v(\cdot,t)}{v_\infty}-1\right\|_{C^2(\overline\om)}\le Ct^{-\theta}\quad\mbox{for all }t\ge 1,
\ee
with a possibly different $\theta$.

Now let us assume that $\Omega$ satisfies the condition \eqref{condition2}. Similar to \eqref{eq:invertibleestimate},  there exists $C>0$ such that for all $\varphi\in H^1_0$ satisfying
\[
-\Delta \varphi-pv_\infty^{p-1}\varphi=f\quad\mbox{in }\Omega,
\]
there holds
\be\label{eq:invertibleestimate2}
\|\varphi\|_{H^1_0(\Omega)}\le C \|f\|_{L^2(\Omega)}.
\ee

As before, on one hand, using the equation of $v_\infty$, we have
\begin{align*}
F(v(t))-F_\infty \le C \|v(\cdot,t)-v_\infty\|_{H^1_0(\Omega)}^2.
\end{align*}
On the other hand,  using the equation of $v_\infty$, we have
\begin{align*}
&\|\Delta v(\cdot, t)+v(\cdot,t)^{p} \|_{L^2(\Omega)}\\
&\ge \left\|\Delta (v(\cdot, t)-v_\infty)+p v_\infty^{p-1}(v(\cdot,t)-v_\infty)\right\|_{L^2(\Omega)}\\
&\quad- \left\|v(\cdot,t)^p-v_\infty^p-p v_\infty^{p-1}(v(\cdot,t)-v_\infty)\right\|_{L^2(\Omega)}\\
&\ge \frac 1C \|v(\cdot, t)-v_\infty\|_{H^1_0(\Omega)}- C \left\| v_\infty ^{\max(0, p-2)}|v(\cdot,t)-v_\infty|^{\min(2, p)}\right\|_{L^2(\Omega)}\\
&\ge \frac 1C \|v(\cdot, t)-v_\infty\|_{H^1_0(\Omega)} -o(1) \|v(\cdot, t)-v_\infty\|_{H^1_0(\Omega)} \quad \mbox{for large }t\\
&\ge \frac 1C \|v(\cdot, t)-v_\infty\|_{H^1_0(\Omega)}\quad \mbox{for large }t,
\end{align*}
where we used \eqref{eq:invertibleestimate2} in the second inequality. The constant $C$ depends only on $n,p,\Omega$ and $u_0$. Combining these two inequalities together, we have 
\[
F(v(t))-F_\infty\le C\|\Delta v(\cdot, t)+v(\cdot,t)^{p} \|_{L^2(\Omega)}^2\le CM_2(t).
\]
It follows that 
\begin{align*}
\frac{\ud }{\ud t}(F(v(t))-F_\infty )= -\frac{1}{p} M_2(t) \le -C (F(v(t))-F_\infty),
\end{align*}
and thus,
\[
F(v(t))-F_\infty  \le Ce^{-\gamma t}
\]
for some $C>0,\gamma>0$.
Hence, the proof of \eqref{eq:M2rate} will give
\[
\int_{T}^\infty M_2(t)^{1/2} \,\ud t \le Ce^{-\gamma t}.
\]
From \eqref{eq:subM2exp}, we obtain that 
\[
\|v(\cdot,t)-v_\infty\|_{L^{p+1}(\Omega)} \le Ce^{-\gamma t}.
\]
Then the proof of \eqref{eq:relativealgebraicsub} gives
\be\label{eq:relativeexpsub}
\left\|\frac{v(\cdot,t)}{v_\infty}-1\right\|_{C^2(\overline\om)}\le Ce^{-\gamma t}\quad\mbox{for all }t\ge 1.
\ee

This finishes the proof of Theorem \ref{thm:subcritical}.
\end{proof}

\appendix

\section{Bubbles interactions} 
In the end of our proof of Proposition \ref{prop:bubble-energy}, we need to calculate and compare the following two quantities:
\begin{align*}
I_1&=\int_{\R^n} \left(\frac{\lda_{1,\nu}}{1+\lda_{1,\nu}^2|x-x_{1,\nu}|^2}\right)^{\frac{n+2}{2}}  \left(\frac{\lda_{2,\nu}}{1+\lda_{2,\nu}^2|x-x_{2,\nu}|^2}\right)^{\frac{n-2}{2}}\,\ud x,\\
I_2&= \int_{\R^n} \left\{\Big(\frac{\lda_{1,\nu}}{1+\lda_{1,\nu}^2|x-x_{1,\nu}|^2}\Big)\lor \Big(\frac{\lda_{2,\nu}}{1+\lda_{2,\nu}^2|x-x_{2,\nu}|^2}\Big)\right\}^{2}\\
&\quad\quad\cdot \left\{\Big(\frac{\lda_{1,\nu}}{1+\lda_{1,\nu}^2|x-x_{1,\nu}|^2}\Big)\land \Big(\frac{\lda_{2,\nu}}{1+\lda_{2,\nu}^2|x-x_{2,\nu}|^2}\Big)\right\}^{n-2} \,\ud x.
\end{align*}
We want to show that under \eqref{eq:bubbles-1}, there holds
\begin{equation}\label{eq:I1I2}
I_1/\sqrt{I_2}=o(1)\quad\mbox{as }\nu\to\infty.
\end{equation}
 
 Recall that $\lda_{1,\nu}\ge \lda_{2,\nu}$.
 
 If $x_{1,\nu}=x_{2,\nu}$, then using the change of variables: $y=\lambda_{2,\nu}x$ and $\Lambda=\lambda_{1,\nu}/\lambda_{2,\nu}$, we have
\begin{align*}
I_1&=\int_{\R^n} \left(\frac{\Lambda}{1+\Lambda^2|y|^2}\right)^{\frac{n+2}{2}}  \left(\frac{1}{1+|y|^2}\right)^{\frac{n-2}{2}}\,\ud y\\
&=\left(\int_{|y|\le \Lambda^{-1}}+\int_{\Lambda^{-1}\le |y|\le 1}+\int_{|y|\ge 1}\right) \left(\frac{\Lambda}{1+\Lambda^2|y|^2}\right)^{\frac{n+2}{2}}  \left(\frac{1}{1+|y|^2}\right)^{\frac{n-2}{2}}\,\ud y\\
&\le C\Lambda^{\frac{2-n}{2}}\\
&\le C \lambda_{1,\nu}^{\frac{2-n}{2}}\lambda_{2,\nu}^{\frac{n-2}{2}},
\end{align*}
and
\begin{align*}
I_2&= \int_{\R^n} \left\{\Big(\frac{\Lambda}{1+\Lambda^2|y|^2}\Big)\lor \Big(\frac{1}{1+|y|^2}\Big)\right\}^{2}\left\{\Big(\frac{\Lambda}{1+\Lambda^2|y|^2}\Big)\land \Big(\frac{1}{1+|y|^2}\Big)\right\}^{n-2} \,\ud y\\
&\ge  \int_{|y|\le 1/\sqrt{\Lambda}} \left(\frac{\Lambda}{1+\Lambda^2|y|^2}\right)^{2}\left(\frac{1}{1+|y|^2}\right)^{n-2} \,\ud y\\
&\ge 2^{2-n} \Lambda^{2-n}\int_{|z|\le \sqrt{\Lambda}} \left(\frac{1}{1+|z|^2}\right)^{2} \,\ud y\\
&\ge c  \Lambda^{2-n} \log\Lambda \quad (\mbox{ since } n\ge 4\mbox{ and }\Lambda\ge 1)\\
&\ge c \lambda_{1,\nu}^{2-n}\lambda_{2,\nu}^{n-2}\ln( \lambda_{1,\nu}/ \lambda_{2,\nu}).
\end{align*}
Since $\lambda_{1,\nu}/ \lambda_{2,\nu}\to\infty$, we have \eqref{eq:I1I2}.

If $x_{1,\nu}\neq x_{2,\nu}$, then we use the following change of variables:
\[
\tilde\lda_{1,\nu}=\lda_{1,\nu}|x_{2,\nu}-x_{1,\nu}|,\ \tilde\lda_{2,\nu}=\lda_{2,\nu}|x_{2,\nu}-x_{1,\nu}|,\ e_\nu=\frac{x_{2,\nu}-x_{1,\nu}}{|x_{2,\nu}-x_{1,\nu}|}.
\]
Then
\begin{align*}
I_1&=\int_{\R^n} \left(\frac{\tilde\lda_{1,\nu}}{1+\tilde\lda_{1,\nu}^2|x|^2}\right)^{\frac{n+2}{2}}  \left(\frac{\tilde\lda_{2,\nu}}{1+\tilde\lda_{2,\nu}^2|x-e_\nu|^2}\right)^{\frac{n-2}{2}}.
\end{align*}
By \eqref{eq:bubbles-1}, we know that
\[
\frac{\tilde\lda_{1,\nu}}{\tilde\lda_{2,\nu}}+\frac{\tilde\lda_{2,\nu}}{\tilde\lda_{1,\nu}} + \tilde\lda_{1,\nu}\tilde\lda_{2,\nu}\to\infty.
\]
Recall that $\lda_{1,\nu}\ge \lda_{2,\nu}$ for all $\nu=1,2,\dots.$ Hence, if $\{\tilde\lda_{1,\nu}\}$ is bounded, then $\tilde\lda_{2,\nu} \to 0 $ and $\frac{\tilde\lda_{1,\nu}}{\tilde\lda_{2,\nu}} \to \infty$.

Case A: $\tilde\lda_{2,\nu}\ge 1$. Then $\tilde\lda_{1,\nu}\to\infty$, and thus
\begin{align*}
I_1&=\int_{B_{1/4}} \left(\frac{\tilde\lda_{1,\nu}}{1+\tilde\lda_{1,\nu}^2|x|^2}\right)^{\frac{n+2}{2}}  \left(\frac{\tilde\lda_{2,\nu}}{1+\tilde\lda_{2,\nu}^2|x-e_\nu|^2}\right)^{\frac{n-2}{2}}\,\ud x\\
&\quad+\int_{\R^n\setminus B_{1/4}} \left(\frac{\tilde\lda_{1,\nu}}{1+\tilde\lda_{1,\nu}^2|x|^2}\right)^{\frac{n+2}{2}}  \left(\frac{\tilde\lda_{2,\nu}}{1+\tilde\lda_{2,\nu}^2|x-e_\nu|^2}\right)^{\frac{n-2}{2}}\,\ud x\\
&\le C \tilde\lda_{1,\nu}^{\frac{2-n}{2}}\tilde\lda_{2,\nu}^{\frac{2-n}{2}}.
\end{align*}

Case B: Otherwise.  Then
\begin{align*}
I_1&\le\tilde\lda_{2,\nu}^{\frac{n-2}{2}}\int_{\R^n} \left(\frac{\tilde\lda_{1,\nu}}{1+\tilde\lda_{1,\nu}^2|x|^2}\right)^{\frac{n+2}{2}}\,\ud x\le C \tilde\lda_{1,\nu}^{\frac{2-n}{2}}\tilde\lda_{2,\nu}^{\frac{n-2}{2}}.
\end{align*}

For $I_2$, we have
\begin{align*}
I_2&=\int_{\R^n} \left\{\Big(\frac{\tilde\lda_{1,\nu}}{1+\tilde\lda_{1,\nu}^2|x|^2}\Big)\lor \Big(\frac{\tilde\lda_{2,\nu}}{1+\tilde\lda_{2,\nu}^2|x-e_{\nu}|^2}\Big)\right\}^{2}\\
&\quad\quad\cdot \left\{\Big(\frac{\tilde\lda_{1,\nu}}{1+\tilde\lda_{1,\nu}^2|x|^2}\Big)\land \Big(\frac{\tilde\lda_{2,\nu}}{1+\tilde\lda_{2,\nu}^2|x-e_{\nu}|^2}\Big)\right\}^{n-2} \,\ud x.
\end{align*}

Let $\va>0$ be sufficiently small. The constant $c$ in the below will be independent of $\va.$

Case A: $\tilde\lda_{2,\nu}\ge 1$. Then $\tilde\lda_{1,\nu}\to\infty$. We split it into two cases:

Case A1: $\tilde\lda_{2,\nu}\ge \va\tilde\lda_{1,\nu}$. Then for $\nu$ large,
\begin{align*}
I_2&\ge \int_{B_{1/2}(e_\nu)} \left\{\frac{\tilde\lda_{2,\nu}}{1+\tilde\lda_{2,\nu}^2|x-e_{\nu}|^2}\right\}^{2}
\cdot \left\{\frac{\tilde\lda_{1,\nu}}{1+\tilde\lda_{1,\nu}^2|x|^2}\right\}^{n-2} \,\ud x\\
&\ge c \tilde\lda_{2,\nu}^{2-n} \tilde\lda_{1,\nu}^{2-n}  \int_{B_{\lda_{2,\nu}/2}} \left\{\frac{1}{1+|x|^2}\right\}^{2} \,\ud x\\
&\ge c \tilde\lda_{2,\nu}^{2-n} \tilde\lda_{1,\nu}^{2-n}(\ln \tilde \lda_{2,\nu})\quad\mbox{if }n\ge 4\\
&\ge c \tilde\lda_{2,\nu}^{2-n} \tilde\lda_{1,\nu}^{2-n}\ln (\va \tilde \lda_{1,\nu})\quad\mbox{if }n\ge 4.
\end{align*}

Case A2: $1\le \tilde\lda_{2,\nu}<\va\tilde\lda_{1,\nu}$. Then for $\nu$ large,
\begin{align*}
I_2&\ge \int_{B_{1/2}\setminus B_{2\sqrt{\va}}} \left\{\frac{\tilde\lda_{2,\nu}}{1+\tilde\lda_{2,\nu}^2|x-e_{\nu}|^2}\right\}^{2}
\cdot \left\{\frac{\tilde\lda_{1,\nu}}{1+\tilde\lda_{1,\nu}^2|x|^2}\right\}^{n-2} \,\ud x\\
&\ge c \tilde\lda_{2,\nu}^{-2}\tilde\lda_{1,\nu}^{2-n} \int_{B_{1/2}\setminus B_{2\sqrt{\va}}} |x|^{4-2n}\,\ud x      \\
&\ge c \tilde\lda_{2,\nu}^{2-n} \tilde\lda_{1,\nu}^{2-n}|\ln \va|\quad\mbox{if }n\ge 4.
\end{align*}

Case B: $ \tilde\lda_{2,\nu}\le 1$. Then $\frac{\tilde\lda_{1,\nu}}{\tilde\lda_{2,\nu}}\to \infty$. We split it into two cases.

Case B1: $\tilde\lda_{1,\nu}\tilde\lda_{2,\nu}\ge \va^{-1}$. Then $\tilde\lda_{1,\nu}\to \infty$. We have
\begin{align*}
I_2&\ge \int_{B_{1/2}\setminus B_{2\sqrt{\va}}} \left\{\frac{\tilde\lda_{2,\nu}}{1+\tilde\lda_{2,\nu}^2|x-e_{\nu}|^2}\right\}^{2}
\cdot \left\{\frac{\tilde\lda_{1,\nu}}{1+\tilde\lda_{1,\nu}^2|x|^2}\right\}^{n-2} \,\ud x\\
&\ge c \tilde\lda_{2,\nu}^{2} \tilde\lda_{1,\nu}^{2-n}\int_{B_{1/2}\setminus B_{2\sqrt{\va}}} |x|^{4-2n}\,\ud x      \\
&\ge c \tilde\lda_{2,\nu}^{n-2} \tilde\lda_{1,\nu}^{2-n} |\ln \va|\quad\mbox{if }n\ge 4.
\end{align*}

Case B2: $\tilde\lda_{1,\nu}\tilde\lda_{2,\nu}\le \va^{-1}$.  Then for large $\nu$,
\begin{align*}
I_2&\ge c \int_{B_{\sqrt{1/(2 \tilde\lda_{1,\nu}\tilde\lda_{2,\nu})}}} \left\{\frac{\tilde\lda_{1,\nu}}{1+\tilde\lda_{1,\nu}^2|x|^2}\right\}^{2}\left\{\frac{\tilde\lda_{2,\nu}}{1+\tilde\lda_{2,\nu}^2|x-e_{\nu}|^2}\right\}^{n-2} \,\ud x\\
&\ge c\tilde\lda_{2,\nu}^{n-2} \int_{B_{\sqrt{1/(2 \tilde\lda_{1,\nu}\tilde\lda_{2,\nu})}}}\left\{\frac{\tilde\lda_{1,\nu}}{1+\tilde\lda_{1,\nu}^2|x|^2}\right\}^{2} \,\ud x \quad(\mbox{since }\tilde\lda_{2,\nu}^2|x-e_{\nu}|^2\le 3)\\
&\ge c \tilde\lda_{1,\nu}^{2-n}\tilde\lda_{2,\nu}^{n-2} \ln (\tilde\lda_{1,\nu}/\tilde\lda_{2,\nu})\quad\mbox{if }n\ge 4.
\end{align*}
Therefore, \eqref{eq:I1I2} holds.




The following calculus lemma was used.
\begin{lem} \label{lem:basic-1}  For $p>2$, $0\le \va\le 1$, we have 
\[
(1+\va)^p \ge 1+\va^p +p \va +c_p \va^2 
\]
and 
\[
(1+\va)^p  \ge 1+\va^p +p \va^{p-1} +c_p \va
\]
for some $c_p>0$ independent of $\va$.

\end{lem}

\bigskip

\noindent T. Jin

\noindent Department of Mathematics, The Hong Kong University of Science and Technology\\
Clear Water Bay, Kowloon, Hong Kong\\[2mm]
 \textsf{Email: tianlingjin@ust.hk}

\bigskip

\noindent J. Xiong

\noindent School of Mathematical Sciences, Laboratory of Mathematics and Complex Systems, MOE\\
Beijing Normal University, Beijing 100875, China\\[2mm]
 \textsf{Email: jx@bnu.edu.cn}

\end{document}